\documentclass[12pt,leqno]{amsart}
\usepackage{graphicx}
\usepackage{geometry}
\usepackage{amsmath}
\usepackage{amssymb}
\usepackage{amsthm}
\usepackage{amsfonts}
\usepackage{mathrsfs}
\usepackage{enumerate}
\usepackage[scriptsize,hang,raggedright]{subfigure}
\usepackage[cmyk]{xcolor}
\usepackage{mathtools}

\usepackage{hyperref}
\usepackage{pdfsync}
\usepackage{dsfont}
\usepackage{color}

\usepackage{cite}
\usepackage{bbm}
\usepackage{multirow}

\allowdisplaybreaks

\definecolor{refkey}{gray}{.45}
\definecolor{labelkey}{gray}{.45}

\usepackage{tikz}
\usetikzlibrary{calc}
\usetikzlibrary{decorations.markings}
\usetikzlibrary{decorations.pathmorphing}
\usetikzlibrary{decorations.shapes}
\usetikzlibrary{shapes,arrows,snakes,shapes.geometric,patterns,fadings}

\numberwithin{equation}{section} 

\newtheorem{theorem}{Theorem}[section]
\newtheorem{proposition}[theorem]{Proposition}
\newtheorem{lemma}[theorem]{Lemma}

\theoremstyle{remark}
\newtheorem{remark}[theorem]{Remark}
\theoremstyle{definition}
\newtheorem{example}[theorem]{Example}
\newtheorem{defi}[theorem]{Definition}
\newcommand{\R}{\mathbb{R}}

\newcommand{\ba}{\begin{array}}
\newcommand{\ea}{\end{array}}
\newcommand{\trns}[1]{\widehat{#1}\,}

\newcommand{\tld}[1]{\widetilde{#1}}
\newcommand{\bthm}{\begin{theorem}}
\newcommand{\ethm}{\end{theorem}}
\newcommand{\bprop}{\begin{proposition}}
\newcommand{\eprop}{\end{proposition}}
\newcommand{\blemma}{\begin{lemma}}
\newcommand{\elemma}{\end{lemma}}
\newcommand{\bexmpl}{\begin{example}}
\newcommand{\eexmpl}{\end{example}}

\newcommand{\beqn}{\begin{equation}}
\newcommand{\eeqn}{\end{equation}}
\newcommand{\beqns}{\begin{equation*}}
\newcommand{\eeqns}{\end{equation*}}

\newcommand{\supp}{\operatorname{supp}}

\newcommand{\pr}{\prime}
\newcommand{\pt}{\partial}

\newcommand{\wlimit}{\rightharpoonup}

\newcommand{\RRR}{\mathbb{R}^3}

\newcommand{\ol}{\overline}

\newcommand{\M}{\mathscr{M}}
\newcommand{\cG}{\mathcal{G}}

\renewcommand{\leq}{\leqslant}
\renewcommand{\geq}{\geqslant}

\newcommand{\TTT}{{\mathbb{T}^3}}

\newcommand{\Izero}{\mathcal{M}_0}
\newcommand{\bigO}{\mathcal{O}}

\definecolor{mygreen}{rgb}{0.1,0.75,0.2}

\newcommand{\N}{\mathbb{N}}

\newcommand{\eps}{\epsilon}

\newcommand{\E}{\mathsf{E}}
\newcommand{\varE}{\mathcal{E}}

\newcommand{\F}{\mathsf{F}}

\newcommand{\Om}{\Omega}

\DeclareMathOperator{\dive}{div}

\DeclareMathOperator*{\dist}{dist}
\DeclareMathOperator{\loc}{loc}
\title[Droplet phase in a NLIP under confinement]{Droplet phase in a nonlocal isoperimetric problem under confinement}

\author{Stan Alama}
\address{Department of Mathematics and Statistics, McMaster University, Hamilton, ON}
\email{alama@mcmaster.ca}
\author{Lia Bronsard}
\address{Department of Mathematics and Statistics, McMaster University, Hamilton, ON}
\email{bronsard@mcmaster.ca}
\author{Rustum Choksi}
\address{Department of Mathematics and Statistics, McGill University, Montr\'{e}al, QC}
\email{rustum.choksi@mcgill.ca}
\author{Ihsan Topaloglu}
\address{Department of Mathematics and Applied Mathematics, Virginia Commonwealth University, Richmond, VA}
\email{iatopaloglu@vcu.edu}

\date{\today}                                        

\begin{document}

\begin{abstract}
We address small volume-fraction asymptotic properties of a nonlocal isoperimetric functional with a confinement term, derived as the sharp interface limit of a variational model for self-assembly of diblock copolymers under confinement by nanoparticle inclusion. We introduce a small parameter $\eta$ to represent the size of the domains of the minority phase, and study the resulting droplet regime as $\eta\to 0$.  By considering confinement densities which are spatially variable and attain a unique nondegenerate maximum, we present a two-scale asymptotic analysis wherein a separation of length scales is captured due to competition between the nonlocal repulsive and confining attractive effects in the energy. A key role is played by a parameter $M$ which gives the total volume of the droplets at order $\eta^3$ and its relation to existence and non-existence of Gamow's Liquid Drop model on $\R^3$.  For large values of $M$, the minority phase splits into several droplets at an intermediate scale $\eta^{1/3}$, while for small $M$ minimizers form a single droplet converging to the maximum of the confinement density.

\end{abstract}

\maketitle

%\setcounter{tocdepth}{1}
%\tableofcontents

%%%%%%%%%%%%%%%%%%%%%%%%%%%%%%%%%%%%%%%%%%%%%%%%%%%%%%%%%%%%%%%%%%%%%%%%%%%%%%%%%%%%%%%%%%%%%%%%%%%
%%%%% INTRODUCTION
%%%%%%%%%%%%%%%%%%%%%%%%%%%%%%%%%%%%%%%%%%%%%%%%%%%%%%%%%%%%%%%%%%%%%%%%%%%%%%%%%%%%%%%%%%%%%%%%%%%
\section{Introduction}\label{sec:intro}

In this paper we study the asymptotic properties of a nonlocal isoperimetric functional with a confinement term. 
This functional appears as the sharp interface limit of a model of diblock copolymer/nanoparticle blend where a large number static nanoparticles serve as a confinement term, penalizing the energy outside of a fixed region.  We choose a scaling regime in which the mass fraction between the two phases tends to zero, but with very strong nonlocal interactions, wherein the minimizing phases resemble small spherical inclusions of one phase in a large sea of the second phase.  This is often called the {\it droplet} regime, in which the sparse $A$-phase can be described effectively by the droplet centers, which we think of as {\it particles} \cite{ChPe2010}.  
The advantage of the droplet regime is to permit the decomposition of the nonlocal effects into {\it self-effects} on the shape of a single droplet and {\it interaction effects} between different particles.

Let $\eta>0$ be a small parameter, which will represent the scale of the radius of the droplets.
We consider periodic configurations $v\in BV(\TTT;\{0,\eta^{-3}\})$, defined on the standard unit torus $\TTT$, which we model as the unit cube $\left[-\frac12,\frac12\right]$ with periodic boundary conditions, and subject to a mass constraint,
$$   \int_{\TTT} v \,dx= M,   $$
for a constant $M>0$.  For a fixed function $\rho\in C(\TTT)$, representing the local density of nanoparticles we analyze the energy functional,
\beqn\label{venergy}
\E_\eta(v):=  \eta\, \int_{\TTT} |\nabla v| \,+\, \eta \, \|v-M\|_{H^{-1}(\TTT)}^2
- \int_{\TTT}  v(x)\rho(x)\, dx 
\eeqn
in the limit as $\eta\to 0$.  The physical background and justification for the choices of the parameters appearing in $\E_\eta$ will be discussed in the Appendix. Here, the first term in the energy denotes the \emph{total variation} of the function $v$ and is defined as
	\[
		\int_{\TTT} |\nabla v| \,:=\, \sup \left\{ \int_{\TTT} v \dive \varphi \,dx \colon \varphi \in C_0^1(\TTT;\RRR), |\varphi(x)| \leq 1 \right\}.
	\]
For characteristic functions $\chi_A\in BV(\TTT;\{0,1\})$, this term computes the perimeter of the interface $\partial A$. The second term, on the other hand, is nonlocal and is defined via  
	\[ 
		\left\|v\, -\,  M\right\|_{H^{-1}(\TTT)}^2 \, 
		=\int_{\TTT}\int_{\TTT} G(x,y)v(x)\, v(y)\, dxdy =
		\, \int_\TTT |\nabla w_v|^2 \, dx,
	\]
where $G(x,y)$ is the mean-zero Laplace Green's function on $\TTT$, and $w_v$ is the solution of $-\Delta w_v = v-M$ on $\TTT$.  

The last term reflects the presence of nanoparticles of density $\rho$ in the sample, which attract the $\{\eta^3v=1\}$ phase of the copolymer.  Thus we expect that droplets accumulate near maxima of the density $\rho$.  For most of the paper we will assume that the density has a unique non-degenerate global maximum at the origin.  
\medskip
	\begin{itemize} 
	\setlength\itemsep{1em}
		\item[(H1)] $\rho\in C(\TTT)$  with $\rho \geq 0$.
		\item[(H2)] $\rho_{\max}:=\rho(0)>\rho(x)$ for all $x\in\TTT\setminus\{0\}$.
		\item[(H3)] $\rho \in C^2(B_r)$ for some $r>0$ and 
			\beqn 
			%\label{eqn:rho_defn}
				\rho(x) = \rho_{\max} - q(x) + o(|x|^2)   \qquad {\rm as} \,\,\, |x|\to 0 \qquad {\rm where} 
				\nonumber
			\eeqn
			\beqn\label{eqn:q_defn} 
				q(x) \, := \frac12\, \sum_{i,j=1}^3 H_{ij} x_i x_j \qquad {\rm for} \,\,\, x = (x_1,x_2,x_3) \in \TTT
			\eeqn
		and $H_{ij}$ are the entries of the Hessian matrix of $-\rho$ given by
			\[
				H_{ij} = -\frac{\pt^2 \rho}{\pt x_i \pt x_j}(0)
				\qquad \text{with} \quad H_{ij} x_i x_j \geq C |x|^2 \quad \text{for some constant } C>0.
			\]
	\end{itemize}
In section~\ref{sec:degenerate_max} we will examine how to treat $\rho$ which have degenerate but homogeneous behavior $\rho_{max}-\rho(x)\sim |x|^q$, $q>2$, at the maximum point. 

We also note that it is not necessary to impose periodic boundary conditions to observe the droplet splitting and confinement described in this paper; for instance, it is both physically and mathematically reasonable to replace the torus $\TTT$ by a smooth bounded domain $\mathcal{D}\subset\R^3$ with Neumann boundary conditions.

\bigskip

\subsection*{Droplet splitting}  In their study of the NLIP under droplet scaling without confinement, (that is, with $\rho(x)\equiv 0$,) Choksi and Peletier \cite{ChPe2010} show that a minimizing sequence $v_\eta$ will decompose into droplets of radius on the order of $\eta$; that is, $v_\eta$ is approximated by a finite sum of weighted Dirac measures, with centers $x_\eta^i\in\TTT$.  The shape of the droplets is determined by the leading order term in an expansion of the energy:  blowing up at scale $\eta$, they arrive at a nonlocal isoperimetric problem (NLIP) on all of $\RRR$.  In the unconfined setting of \cite{ChPe2010}, the location of the centers $x_\eta^i$ is determined by the next order term in the energy expansion, a Coulomb-like repulsion arising from the nonlocal term.  The result is convergence of minimizers to a discrete pattern of droplets spread out on $\TTT$, where the interdroplet distances are bounded below.  

For our energy \eqref{venergy} with nonconstant density $\rho$ we observe a different qualitative picture. The confinement term in $\E_\eta$ will be felt at the level of particle interactions, drawing the centers $x_\eta^i$ towards the global maximum of the nanoparticle density $\rho(x)$.
This is the essential difference between our problem and that of \cite{ChPe2010}:  as the droplet centers will be drawn towards a single point we must isolate individual droplets at a much smaller scale.  The interaction terms in the energy expansion of \cite{ChPe2010} are of order one, and their method suffices to discern droplets at arbitrarily small distance of order one.  This is not enough for our setting, where the second-order term in the energy expansion is of much smaller order than the order of approximation of the first-order term in \cite{ChPe2010}.  Thus, we must develop a more refined method of isolating the droplets and calculating the interaction distance on a much smaller scale.

\medskip

To motivate our main result, let us consider an ansatz for $v_\eta$, which will form the basis of the upper bound on the energy (see Lemma~\ref{lem:upper}.)  We will suppose that minimizers form $n$ droplets, centered at points $x_\eta^i=\delta\, p_i$ with fixed $p_1,\dots,p_n$ and $\delta=\delta(\eta)\to 0$, reflecting the tendency of the droplets to approach the maximizer of $\rho(x)$ at a particular rate (to be determined.)  We thus define an admissible test configuration,
\beqn\label{eq:ansatz}
   \nu_\eta := \sum_{i=1}^n  \eta^{-3}w_i\left( {x-\delta p_i\over \eta}\right), 
\eeqn
where $w_i\in BV(\RRR;\{0,1\})$ are characteristic functions representing the droplet sets, blown up by scale $\eta$ to $\RRR$.  We define 
$m^i:=\int_{\RRR} w_i\,dx$, and note that the constraint $\int_{\TTT} \nu_\eta\,dx=M$ forces
$\sum_{i=1}^n m^i =M$.

We may then evaluate  $\E_\eta(\nu_\eta)$ asymptotically, expressing the $H^{-1}$-norm in terms of the Green's function, $G(x,y)\sim 1/4\pi|x-y|$, as the points coalesce.  Assuming the distance between droplet centers $\delta(\eta)\gg\eta$, a back-of-the-envelope calculation yields,
\beqn\label{eq:enansatz}
\E_\eta(\nu_\eta) \simeq \sum_{i=1}^n \left[ \int_{\R^3} |\nabla w_i| 
  + \left\| w_i\right\|^2_{H^{-1}(\R^3)}\right] - M\rho_{\max}
  + \left[ {\eta\over\delta}\sum_{\substack{i,j=1\\ i\neq j}}^n
       \frac{m^i\,m^j}{{4\pi}|p_i-p_j|} + \delta^2\sum_{i=1}^n m^i q(p_i)\right].
\eeqn
Optimizing the right-hand expression over $\delta>0$ (holding all other quantities fixed) we predict the droplet separation scale  $\delta=\mathcal{O}(\eta^{1/3})$ (see Figure~\ref{fig:attr_to_origin}).  
Thus, we expect droplet separation to occur on {\bf two different scales} in the length and in energy.  The droplets themselves have characteristic length $\eta$, and contribute to the energy at order one, while the pattern they form is observed on the much larger length scale $\eta^{1/3}$, which affects the energy to order $O(\eta^{2/3})$.

The first sum in the expansion is the energy for Gamow's Liquid Drop model, and this indicates how droplet shape is determined via minimization.  Droplet profiles $z_i$ should be minimizers for the nonlocal isoperimetric problem on $\R^3$:
	\beqn\label{eqn:nlip}
 e_0 (m) \,: = \, 
\inf \left\{ \int_{\R^3} |\nabla z| \, + \,
\| z\|_{H^{-1} (\R^3)}^2 
 \,\,  : \,  \, z \in BV (\R^3, \{0,1\}), \,\, 
\int_{\R^3} z \,dx \, = \, m \right\},
\eeqn
where 
\[ \| z\|_{H^{-1} (\R^3)}^2 \, = \, \int_{\R^3} \int_{\R^3} \frac{z(x) \, z(y)}{4 \pi |x - y|} \, dx \, dy. \]
This problem has been extensively studied, in the form above and in various generalizations; see \cite{BoCr14,Ju2014,FrLi2015,KnMu2013,KnMu2014,LO1,LO2, FKN2}.
In  \cite{ChPe2010} it was conjectured that there 
exists a critical mass $m_c$ such that minimizers are balls 
for $m \leq m_c$ and  fail to exist otherwise. 
The results will play an important role in this article  and we summarize them in the next theorem, which may be found in \cite{KnMu2014} (see also \cite{LO1}):
\begin{theorem}[Kn\"upfer-Muratov]\label{thm:e0}
 There exist three constants $0<m_{c_0} \leq m_{c_1} \leq m_{c_2}$ such that the following hold.
 	\begin{itemize}
 	 \item[(i)] If $m \leq m_{c_1}$ then $e_0(m)$ admits a minimizer. If $m \leq m_{c_0}$ the minimizer is unique (up to translation), and is given by the ball of volume $m$. The ball of volume $m$ ceases to be the minimizer of $e_0(m)$ for $m>m_{c_0}$.	 
 	  \item[(ii)] If $m>m_{c_2}$ then $e_0(m)$ does not admit a minimizer.
 	\end{itemize}
\end{theorem}
\noindent  
To date, it remains an open problem to prove or disprove whether any (or all) of the constants $m_{c_i}$, $i=0,1,2$, above are pairwise equal.  

\medskip

Returning to the energy expansion of the ansatz \eqref{eq:enansatz}, we may now complete the heuristic picture of minimizers of $\E_\eta$ in terms of the Gamow functional.  Defining the ``good'' set of $m$ for which the nonlocal isoperimetric problem attains a minimizer, 
\[ 
		\Izero \,:=\, \left\{ m>0 \colon e_0(m) \text{ admits a minimizer} \right\}, 
\]
we observe that when the total mass $M\in\Izero$ (for instance, when $M\leq m_{c1}$ is small enough,) there is no need for splitting, and we anticipate that minimizers $v_\eta$ remain connected as $\eta\to 0$.  However, if $M\not\in\Izero$ then we infer that according to \eqref{eq:enansatz} minimizers will split into droplets with mass $m^i\in\Izero$, for which there is an optimal droplet blowup function $z_i\in BV(\RRR;\{0,1\})$, and with $\sum_{i=1}^n m^i=M$.  
That is, the collection of droplet masses must lie in the special set,
\[ 
		\mathcal{M} \,:=\, \left\{ \{m^i\}_{i=1}^n \colon \ n\in\N, \ m^i \geq 0, \ \sum_{i=1}^n m^i = M, \text{ and }e_0(m^i)\text{ admits a minimizer for each }i \right\}. \]
These droplets will then arrange themselves in a pattern with separation distance $\delta(\eta)=O(\eta^{1/3})$, so as to minimize the interaction energy formed by combining Coulomb repulsion and confinement,
$$   \F_{m^1,\dots,m^n}(p_1,\dots,p_n):= 
\sum_{\substack{i,j=1\\ i\neq j}}^n
       \frac{m^i\,m^j}{{4\pi}|p_i-p_j|} + \sum_{i=1}^n m^i q(p_i),
$$
with $q(x)$ the Hessian of $\rho(x)$ near $x=0$, defined in (H3).  (See Figure~\ref{fig:attr_to_origin}).
%% %%

\begin{figure}[ht!]
		\begin{center}
			\includegraphics[height=6cm]{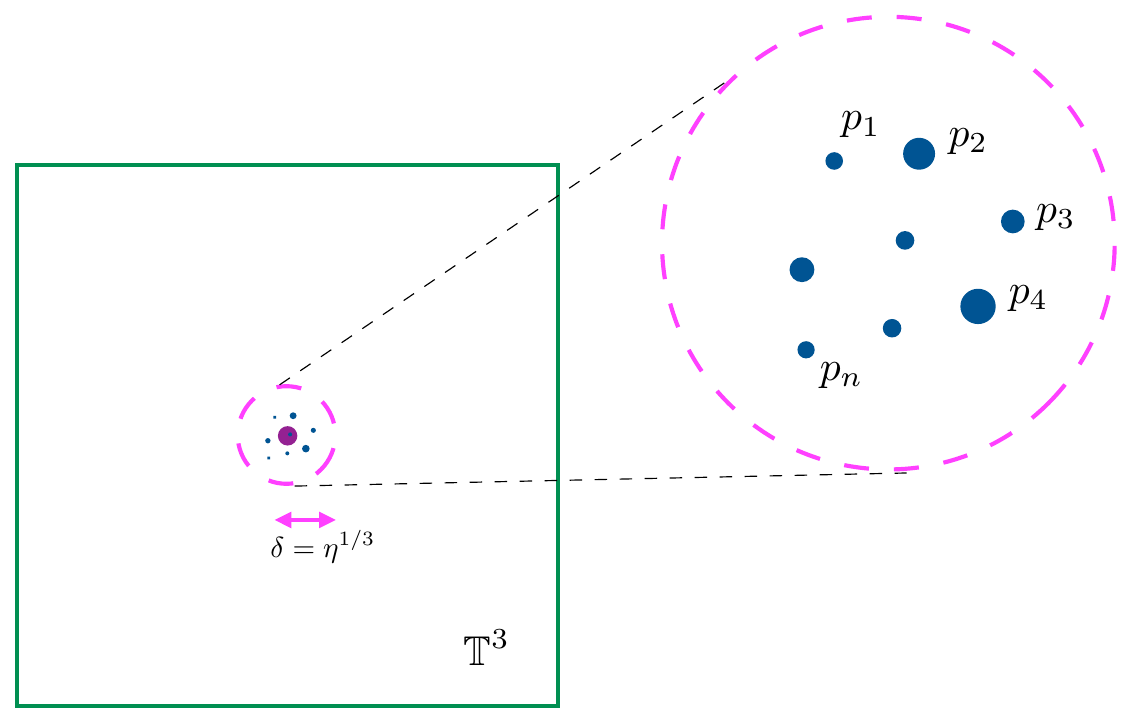}
		\end{center}
		\caption{
        The attraction to the origin and scaling at the rate $\delta=\eta^{1/3}$.
     }
   		\label{fig:attr_to_origin}
\end{figure} 

Our main result is to confirm this expected behavior by means of a precise asymptotic expansion of the energy of minimizers.  We prove:

\begin{theorem}\label{thm:minimizers}
Let $v_\eta$ be minimizers of $\E_\eta$ in $BV(\TTT, \{0,\eta^{-3}\})$ with
$\int_{\TTT}v_\eta \, dx=M$.  
\begin{itemize}
\item[(i)] For any $r>0$, $\supp v_\eta\subset B_r(0)\subset\TTT$ for all sufficiently small $\eta>0$.
\item[(ii)]  If $M\in\Izero$, then there exists a subsequence of $\eta\to 0$ and points $y_\eta\in\TTT$ with $|y_\eta|\leq C\eta^{1/2}$ such that  
$$   v_\eta - \eta^{-3} z_M\left({x-y_\eta\over\eta}\right) \longrightarrow 0\quad\text{ in $L^1(\TTT)$, }
$$
where $z_M$ attains the minimum $e_0(M)$.
%
% $w\in BV(\R^3, \{0,1\})$ with $\int_{\R^3} w=M$ and points $x_\eta$ with $|x_\eta|\leq C\eta^{1/2}$
%so that $v_\eta- \eta^{-3}w({\cdot-x_\eta\over\eta})\to 0$ in $L^1(\TTT)$, and $\E_\eta(v_\eta)= e_0(M)-M\rho_{\max} + \mathcal{O}(\eta)$.
\item[(iii)]  If $M\not\in\Izero$, then, there exists a subsequence of $\eta\to 0$, $n\in\N$, $\{m^i\}_{i=1}^n\in\mathcal{M}$, and distinct points $\{x^1_\eta,\dots,x^n_\eta\}$ such that:
\beqn\label{eq:split0}
 v_\eta - \sum_{i=1}^n \eta^{-3} w_i\left({x - x^i_\eta\over\eta}\right) \longrightarrow 0, \quad
 \text{in $L^1(\TTT)$},
\eeqn
where $w_i\in BV(\R^3, \{0,1\})$ attains the minimum in $e_0(m^i)$, $i=1,\dots,n$.  In addition,
\begin{gather*}  
\eta^{-\frac13}x_\eta^i \longrightarrow x_i \in\RRR, \quad\text{and}
\\  
\E_\eta(v_\eta) = e_0(M) - M\rho_{\max} +   \eta^{2/3}\left[
  \sum_{i=1}^n m^i q(x_i) + \frac{1}{4\pi}\sum_{\substack{i,j=1\\ i\neq j}}^n \frac{m^i\,m^j}{|x_i-x_j|}\right] +o\left(\eta^{2/3}\right).
\end{gather*}
Moreover, the expression in brackets above is minimized by the choice of points $\{x_1,\dots,x_n\}$ given the values $\{m^i\}_{i=1}^n\in\mathcal{M}$.
\end{itemize}
\end{theorem}

\begin{remark}\label{rho_remark}
As noted above, we may also consider $\rho$ which attain a unique degenerate maximum, assuming a different homogeneity around the maximum point.  The analysis will be essentially the same, except for the scale of droplet separation and the form of the interaction energy;  see Section~\ref{sec:degenerate_max}.  In case $\rho$ attains its global maximum at a finite number of points we expect the same general behavior near one or several of the maximizers, but the distribution of mass around each one depends subtly on the shape of the function at each point.
\end{remark}

Part (ii) of Theorem~\ref{thm:minimizers} implies that when $M\in\Izero$ there is no splitting of the droplet:  apart from a possible drift of $\mathcal{O}(\eta^{1/2})$ away from the origin, it is essentially a rescaling of the minimizer of the nonlocal isoperimetric problem at scale $\eta$.  It is known \cite{LO1, KnMu2014} that the $e_0(M)$ minimizers (when they exist) have connected compact support. 
Thus, the theorem asserts that splitting can only occur if the $e_0(M)$ problem fails to achieve a minimum and that the separation must be $\mathcal{O}(\eta^{1/3})$ and at no other scale.  For clarity, we note that for $\eta$ sufficiently small, the sum in \eqref{eq:split0} is disjointly supported in $B_{\frac14}(0)\subset\RRR$, and may thus be treated as a function in $BV(\TTT;\{0,\eta^{-3}\})$. 

%An important issue is to show that the number of droplets $n$ is finite. As in \cite{ChPe2010} this is a delicate issue which depends crucially on the second-order term in the expansion (see Lemma~\ref{lem:compactness}).

The interaction energy $\F_{m^1,\dots,m^n}$, which determines the location of the droplets at $\eta^{1/3}$-scale, is (for fixed mass distribution $\{m^i\}_{i=1}^n$) coercive and clearly attains a minimum among all possible configurations $\{x_i\}_{i=1}^n$ in $\R^3$.  The two-dimensional analogue of $\F_{m^1,\dots,m^n}$, with logarithmic repulsion, was studied in \cite{GuSh}, but in the case of equal masses.  (The motivation behind \cite{GuSh} came from Ginzburg--Landau vortices, which have quantized flux playing the role of the mass in $\F_{m^1,\dots,m^n}$.)  We know of no studies of $\F_{m^1,\dots,m^n}$ which address the fundamental question of how to optimally choose both $\{m^i\}_{i=1}^n$ and $\{x_i\}_{i=1}^n$ to minimize $\F_{m^1,\dots,m^n}$.

We also note that similar concentration or coalescence phenomena appear in various examples of singularly perturbed variational problems, notably in low density phases in a uniformly charged liquid \cite{KnMuMa2016} and for vortices in the 2D Ginzburg--Landau model of superconductivity \cite{SaSe2007}. See, for example, Chapter~7 of \cite{SaSe2007}, in which finitely many vortices concentrate at specific points in a sample subjected to an external magnetic field $h_{ex}$ of order $\ln\kappa$, where $\kappa$ is the Ginzburg--Landau parameter.  As above, there are two length scales:  the radii of the vortex cores are given by $\kappa^{-1}$, and the distance between coalescing vortices is of order $h_{ex}^{-1/2}$.  Note in particular the similarity of their renormalized energy (11.2) and our interaction energy $\F_{m^1,\dots,m^n}$.

%As noted above, it is both physically and mathematically natural to replace the periodic boundary conditions imposed here by a Neumann problem in a bounded domain in $\R^3$.  In the Neumann setting, one should observe the same behavior for interior droplets forming near the maximum of $\rho$, but it would also be possible for droplets to gravitate to the boundary.

Finally, attractive-repulsive nonlocal interaction energies of the form $\F_{m^1,\dots,m^n}$ have attracted much interest lately in connection with models of pattern formation in biological aggregations such as swarming or collective behavior of multi-agent systems \cite{BrechtUminskyKolokolnikovBertozzi,ChafaiGozlanZitt,Chuang_etal,Morale:Capasso}. In these models the particles can be approximated by a density distribution leading to a continuum-level energy, and minimization with respect to these densities yields challenging variational and geometric problems (see e.g. \cite{BCT,FrLi2016}).

Our paper is organized as follows: In the appendix, we describe the physical origins of the model as well as the derivation of the droplet scaling regime which we will use in the paper.  Section~\ref{sec:second_ord} describes the concentration structure of minimizers.  The most important result there is a key concentration-compactness lemma (Lemma~\ref{lem:compactness}) which gives the framework for our result, and is based on a recent  compactness result for sets of bounded perimeter due to Frank and Lieb \cite{FrLi2015} which 
can conveniently be used to decompose the support of a minimizing sequence into sets of diameter $O(\eta)$. In Section~\ref{sec:regularity} we use the regularity of minimizers to refine the concentration lemma and obtain a lower bound with a precise error estimate of order $o(\eta^{2/3})$, needed to isolate the interaction terms in the energy.    In Section~\ref{sec:proof_of_conv} we prove Theorem~\ref{thm:minimizers} by matching upper and lower bounds on the energy. Finally in Section~\ref{sec:degenerate_max} we discuss second-order approximations for degenerate confining penalization densities $\rho$ which have degenerate maxima at a single point.

\section{The Concentration Structure of Minimizers}\label{sec:second_ord}

  We begin by  defining the first-order limit functional,  
\[ 
		\E_0(v) \, := \, \begin{dcases*}
									\sum_{i=1}^{\infty} e_0(m^i)\, -\, m^i\rho(x_i) & \quad if $v=\sum_{i=1}^\infty m^i\delta_{x_i}$, $\{x_i\}$ distinct, and $m^i \geq 0$, \\
									+\infty								& \quad otherwise, 
							 \end{dcases*}
\]
where $e_0$ is defined by \eqref{eqn:nlip}. 
The following first-order Gamma convergence result follows directly from \cite{ChPe2010}, since the confinement term is a continuous perturbation of the functional in  \cite[Theorem 4.3]{ChPe2010}.
\bthm[First-order limit]\label{thm:first_ord_limit}
		The sequence of functionals $\E_\eta$ $\Gamma$-converge to $\E_0$ as $\eta \to 0$ in the space of Radon measures with respect to the weak-* convergence. 
	\ethm

\begin{remark}\label{rem:first_ord_general_rho}
	Theorem \ref{thm:first_ord_limit} holds true for any confinement term defined via a penalizing measure $\rho \in C(\RRR)$ and is not special to those which satisfy (H1)--(H3).
\end{remark}

\begin{remark}
A standard consequence of $\Gamma$-convergence, combined with the compactness of a sequence of minimizers, is that minimizers of $\E_\eta$ converge weakly in the space of Radon measures to a minimizer of the limiting energy $\E_0$. Hence, using the simple fact that $e_0(M) \leq \sum_{i=1}^\infty e_0(m^i)$ (cf. \cite{ChPe2010}) we may conclude that for any family of mass-constrained minimizers of $\E_\eta$,
		\beqn \label{eqn:min_conv_min}
	v_\eta\,  \wlimit \,  M  \delta_0  \qquad {\rm and} \qquad 		\lim_{\eta\to 0} \E_\eta(v_\eta) \,= \, e_0(M) - M\rho_{\max}
		\eeqn
	where $\rho_{\max}:=\rho(0)$.
\end{remark}

\medskip

As Theorem \ref{thm:first_ord_limit} shows the first-order approximation does not include the nonlocal non-self interaction effects of the energy $\E_\eta$. To capture these effects we will look at the second-order approximation. This limit will depend on the specific form of the penalizing measure $\rho$ as we will normalize the energy $\E_\eta$ by subtracting the ground state of the first-order limit $\E_0$. We just noted that as $\eta \to 0$, mass constrained minimizers of $\E_\eta$ concentrate at the origin.  How  the sequence  collapses  to the origin depends on the size of the mass constraint $M$. 
When $M$ is sufficiently large so that a minimizer of $e_0$ fails to exist, we show that minimizers of $\E_\eta$ split at a scale larger than $\eta$ as $\eta\to 0$, while for small enough $M$ we show that no splitting may occur. 

To this end, recall the notation,

	\[
		\Izero \,:=\, \{ M>0 \colon e_0(M) \text{ admits a minimizer}\}.
	\]
It is clear from  Theorem \ref{thm:e0} that $\Izero$ is not empty. Moreover, $M<m_{c_1}$ is a sufficient condition for $M\in\Izero$ and $M>m_{c_2}$ is sufficient for $M \not\in \Izero$; however, the necessity of either condition is an open problem.

\bigskip

As a first illustration of the role played by $\Izero$, we state an ensemble of upper bounds on the minimum energy of $\E_\eta$:
\begin{lemma}\label{lem:upper}
For any $n\in\N$, $\{p_i\}_{i=1}^n$ distinct fixed points in $\R^3$, and $\{m^i\}_{i=1}^n\in\mathcal{M}$, we have
\beqn\label{eq:upper2}
\min_{\int_{\TTT} v =M}  \E_\eta(v)  
\leq \left(e_0(M)-M\rho_{\max}\right) + \eta^{2/3}\left[
\sum_{i=1}^n m^i q(p_i) + \frac{1}{4\pi}\sum_{\substack{i,j=1\\ i\neq j}}^n \frac{m^i\,m^j}{|p_i-p_j|}\right] + o(\eta^{2/3}).
\eeqn
Moreover, if $M\in\Izero$, then
\beqn\label{eq:upper1}
\min_{\int_{\TTT} v \,dx=M}  \E_\eta(v) \leq \left(e_0(M)-M\rho_{\max}\right) + \mathcal{O}(\eta).
\eeqn
\end{lemma}

Thus, if $M\in\Izero$ the upper bound on the energy is vastly improved.  In fact, we expect that fewer droplets (smaller $n$) generally should yield a smaller upper bound, and so energy minimization should split the mass $M$ into the smallest number of pieces which permit $\{m^i\}_{i=1}^n\in\mathcal{M}$.  However, we are not aware of any such results concerning the minimization of $\F_{m^1,\dots,m^n}$ with respect to the masses $m^i$.

The proof of the upper bound follows along the lines of the heuristic argument presented in the Introduction, and we defer it to Section~\ref{sec:proof_of_conv}.

\medskip

As usual, the difficulty in verifying the desired structure of energy minimizers lies mostly in developing a matching lower bound on the energy.  The first and crucial step is obtaining a compactness result which establishes the existence of points in $\TTT$ which are separated by a scale much larger than $\eta$ apart, so that weighted Dirac-delta measures at these points approximate $v_\eta$. This result also gives the existence of components of $\supp v_\eta$ whose supports are $\eta$-rescalings of minimizers of the nonlocal isoperimetric problem, and thus achieve the minimum of the first-order energy $\E_0$. In addition, we establish that there can only be finitely many distinct components for minimizers of $\E_\eta$, and in case there is a unique component we obtain a stronger convergence result:

\blemma \label{lem:compactness}
	For each $\eta>0$ let $v_\eta$ be a minimizer of $\E_\eta$ with $\int_{\TTT} v_\eta\,dx=M$.  Then there exists a subsequence $\eta\to 0$, $n\in\N$, $\{m^i\}_{i=1}^n\subset (0,\infty)$, $\{ x_\eta^i \}_{i=1}^n \subset \TTT$ and functions $w_\eta^i \in BV(\TTT;\{0,1/\eta^3\})$ with $\|w_\eta^i\|_{L^1(\TTT)}=m^i+o(1)$ as $\eta\to 0$ such that for $n \geq 2$
\begin{gather}  
		{|x_\eta^i - x_\eta^j| \over\eta}\to\infty \text{ for every } i \neq j,\\
		\left\|v_\eta - \sum_{i=1}^n w_\eta^i \right\|_{L^1(\TTT)}\to 0 \\ 
		e_0(m^i) \text{ is attained for each }i=1,2,\ldots,n, \quad \text{ and }\quad e_0 (M) = \sum_{i=1}^n e_0 (m^i),  \label{eqn:e0_linearity}
\\ \label{eqn:o1lower}		
\liminf_{\eta\to 0}\E_\eta(v_\eta) \geq \liminf_{\eta\to 0} \sum_{i=1}^n \E_\eta \left( w_\eta^i \right)\geq  \sum_{i=1}^n e_0(m^i) - M\rho_{\max}. 
		\end{gather}
Moreover, if $n=1$, then $M\in\Izero$, and there exist points  $x_\eta\in\TTT$ such that 
\beqn\label{eq:one}
v_\eta-\eta^{-3} z_M\left({x-x_\eta\over \eta}\right) \longrightarrow 0\quad \text{in $L^1(\TTT)$ as $\eta\to 0$,}
\eeqn
 where $z_M$ attains the minimum $e_0(M)$.
\elemma

We remark that the last part of the lemma already indicates that splitting is only to be expected when the $e_0(M)$ problem has no minimizer.  Indeed, the second-order behavior is very different depending on whether $M\in\Izero$ or not.  However, it will be more convenient to use the dichotomy suggested by Lemma~\ref{lem:compactness} and separate the cases $n=1$ and $n\geq 2$ determined above; the full connection to $\Izero$ will only be apparent later on.

\medskip

The proof of Lemma~\ref{lem:compactness} is based on blowing up the minimizing sets $A_\eta=\supp v_\eta$ at order $\eta$.  When we blow up the minimizing set $A_\eta$ we are effectively thinking of it as a subset of $\RRR$, and are able to utilize a technical concentration-compactness result for sets of finite perimeter by Frank and Lieb \cite{FrLi2015}, but this can only work locally.  As the balls $B_r(x_0)\subset\TTT$ with radius $0<r<\frac12$ lift isometrically to $\RRR$, blowup makes sense for sets defined in these balls.   For $x_0\in\TTT$ and $0<r<\frac12$, we may thus define a family of diffeomorphisms,
$$  \left\{\begin{gathered}
  \Phi_{\eta,r}(\cdot;x_0): \ B_r(x_0)\subset\TTT \to B_{r\over\eta}(0)\subset\RRR,  \\
  x\mapsto \Phi_{\eta,r}(x;x_0)= {x-x_0\over\eta},
  \end{gathered}\right.
$$
where the last quantity is lifted to the covering space $\RRR$.  Thus, the map $\Phi_{\eta,r}(x;x_0)$ performs a blowup of $B_r(x_0)$, viewed as a subset of $\RRR$.  For simplicity, we will often abuse notation and write $\Phi_{\eta,r}(S;x_0)= \eta^{-1}(S-x_0)$ for the blowup to $\RRR$ of a set $S\subset B_r(x_0)\subset\TTT$, and for $U\in B_{{r\over\eta}}(0)\subset\RRR$, we write
$\eta U + x_0=\Phi^{-1}_{\eta,r}(U;x_0)\subset\TTT$.

%We use this lemma inductively 
%to find the constants $m^i$, the points $x_i$ and sets $\Omega_i$ with $|\Omega_i | = m_i$ which minimizes $e_0 (m^i)$. 
%To this end, we start by considering $\Omega^\pr_\eta$, the zoomed in support of a minimizing sequence $v_\eta$ at the  $\mathcal{O} (1)$ scale. 
%We  use Part 1 of Lemma \ref{lem:compactness_FrLi} to show that there exist  translation points $y_\eta^1$ ($x_\eta^1$ at  the $\mathcal{O} (\eta)$-scale) such that the sets $\Omega^\pr_\eta$ translated by $y_\eta^1$ converge locally in $L^1$ to a set $\Omega_1$. We call the measure of this set $m_1$ which constitutes the first droplet.  If all the mass has been captured in the first step, $m^1=M$, the process terminates and we show $M\in\Izero$.  If $m^1<M$, we can continue with Part 2 of Lemma \ref{lem:compactness_FrLi}. Here we break up the sets $\Omega^\pr_\eta$ into two pieces, one which tends to $\Omega_1$ and the other which approaches $0$ in $L^1_{\loc}$. This latter part pertains to the remaining droplets. We now repeat the first step on the support of the remaining mass to find $m^2$,  $x_\eta^2$ and $\Omega_2$, and so forth. This process may end at a finite number of iterations. In fact, it ends as soon as we reach a point where the $m^i$ sum to $M$. 

\begin{proof}[Proof of Lemma \ref{lem:compactness}]
For any $\eta>0$ let $v_\eta\in BV(\TTT;\{0,1/\eta^3\})$ with $\int_{\TTT}v_\eta\,dx=M$ be a minimizer of  $\E_\eta$. Such a minimizer exists by the direct method in the calculus of variations.  By the first-order limit, Theorem~\ref{thm:first_ord_limit}, the energy is bounded above, and we have that $\E_\eta(v_\eta)\leq e_0(M)-M\rho_{\max} + o(1)$.

\noindent \textit{\textbf{Step 1.}}  We restrict to a neighborhood of zero. \  
Write 
$v_\eta = \eta^{-3}\chi_{A_\eta}$, with $A_\eta\subset\TTT$.  Since by (H3) zero is a strict local maximum of $\rho$, and thus we may find $\eps>0$ and $\delta\in (0,\frac18)$ so that 
$$\rho(x)<\rho_{\max}-\eps \quad \text{ for all $x\in\TTT\setminus B_\delta(0)$.}
$$
We now claim that for this $\delta$,  $\eta^{-3}|A_\eta\setminus B_\delta(0)|\to 0$.  Indeed, assume the contrary, so there is a value $>0$ such that along some subsequence $\eta\to 0$, $\eta^{-3}|A_\eta\setminus B_\delta(0)|\geq>0$. However, in that case,
\begin{align*}
\int_{\TTT} \rho\, v_\eta\, dx & = \eta^{-3} \int_{A_\eta\cap B_\delta(0)} \rho\, dx
   +  \eta^{-3}\int_{A_\eta\setminus B_\delta(0)} \rho\, dx \\
   &\leq \rho_{\max}\eta^{-3} |A_\eta\cap B_\delta(0)| 
     + (\rho_{\max}-\eps) \eta^{-3}|A_\eta\setminus B_\delta(0)| \\
     &\leq \rho_{\max}M - \eps <\rho_{\max}M.
\end{align*}
Using the lower bound from \cite{ChPe2010} on the first two terms in the energy  together with the upper bound from Theorem~\ref{thm:first_ord_limit},
$$  e_0(M) - M\rho_{\max} \geq \lim_{\eta\to 0} \E_\eta(v_\eta) \geq e_0(M) - \rho_{\max}M + \eps,  $$
which gives the desired contradiction.  

We next show we make only a very small error in the energy of minimizers by restricting to a ball around the origin.  Indeed, following the proof of Lemma~2.2 of \cite{FrLi2016}, for each $\eta$ there exists a radius $\delta_\eta\in (2\delta,3\delta)$ so that the surface area of the intersection,
$$   \mathcal{H}^2(\partial B_{\delta_\eta}(0)\cap A_\eta) \longrightarrow 0.  $$
(The argument involves noting that $|A_\eta\cap (B_{3\delta}\setminus B_{2\delta})|\to 0$ and using Fubini's theorem in spherical coordinates.)  Define the sets
$$  \Om_\eta := A_\eta\cap B_{\delta_\eta}(0), \qquad  \Xi_\eta:= A_\eta\setminus B_{\delta_\eta}(0),  $$
and functions
$$   v_\eta^0:= \eta^{-3}\chi_{\Om_\eta}.  $$
Then, we may conclude that
\begin{align*}
 \int_{\TTT} |\nabla \chi_{A_\eta}| &\leq \int_{\TTT} |\nabla\chi_{\Om_\eta}| + \int_{\TTT} |\nabla\chi_{\Xi_\eta}| \\
 &\leq \int_{\TTT} |\nabla\chi_{A_\eta}| + 
       2 \mathcal{H}^2(\partial B_{\delta_\eta}(0)\cap A_\eta)+o(1)\leq \int_{\TTT} |\nabla\chi_{A_\eta}| + o(1),  
\end{align*}
and therefore
$$   \liminf_{\eta\to 0} \int_{\TTT}|\nabla\chi_{A_\eta}| \geq \liminf_{\eta\to 0} \int_{\TTT}|\nabla\chi_{\Om_\eta}|,\quad\text{and}\quad
          \lim_{\eta\to 0} \left[\int_{\TTT} |\nabla\chi_{A_\eta}| - \int_{\TTT} |\nabla\chi_{\Om_\eta}| - \int_{\TTT} |\nabla\chi_{\Xi_\eta}|\right]=0.
$$   
Furthermore, $\eta^{-3}|\Xi_\eta|\to 0$ (globally on $\TTT$) implies that 
$\| v_\eta-v^0_\eta\|_{L^1(\TTT)}\to 0$, and 
$$  \int_{\TTT}\int_{\TTT} G(x,y)\, v_\eta(x)\, v_\eta(y)\, dxdy =
\eta^{-6}\int_{A_\eta}\int_{A_\eta} G(x,y)  \, dxdy =
\eta^{-6}  \int_{\Om_\eta}\int_{\Om_\eta} G(x,y)  \, dxdy
 + o(1), $$ 
and
$$  \int_{\TTT} v_\eta\rho\, dx = \eta^{-3}\int_{A_\eta} \rho\, dx = \eta^{-3}\int_{\Om_\eta} \rho\, dx +o(1)=\int_{\TTT} v_\eta^0\rho\, dx +o(1),  $$
and hence we have
\beqn\label{eq:Mo1}
 \E_\eta(v_\eta) \geq \E_\eta(v_\eta^0) + o(1), 
\quad\text{and \quad }  \int_{\TTT} v_\eta^0\,dx = \eta^{-3}|\Om_\eta|=M-o(1).
\eeqn
Effectively, we may consider $v^0_\eta$ to be our minimizing sequence, with support in $B_{3\delta}(0)\subset\TTT$ (although it only satisfies the mass constraint to within $o(1)$.)  However, it will often be convenient to treat the  truncated support $\Om_\eta\subset B_{3\delta}(0)\subset\RRR$ as belonging to the Euclidean space $\RRR$.

\medskip

\noindent \textit{\textbf{Step 2.}}  Finding a concentrating set.  We blow up on the scale of $\eta$, to obtain 
	\[
		\trns{\Om}_\eta \,:=\, \eta^{-1}\Om_\eta=\Phi_{\eta,3\delta}(\Om_\eta;0)\subset\RRR.
	\]
We note that in fact $\trns{\Om}_\eta\subset B_{3\delta\over\eta}(0)$, and  by \eqref{eq:Mo1},
 $|\trns{\Om}_\eta|=\eta^{-3}|\Om_\eta|=M-o(1)$.
For a set $\widehat\Om\subset\RRR$, we define the energy
$$\cG(\widehat\Om):= \int_{\RRR}|\nabla\chi_{\widehat\Om}| 
    + \int_{\widehat\Om}\int_{\widehat\Om} {1\over 4\pi |x-y|} \, dxdy,
$$
corresponding to the Liquid Drop model (see cf. \cite{KnMu2014,FrLi2015} or \eqref{eqn:nlip}).  For $x,y\in B_{3\delta}(0)\subset\TTT$ we may express the Green's function on the torus locally,
\beqn\label{eq:Green1}
		G(x,y) = \frac{1}{4\pi|x-y|} + g(x-y),
\eeqn
in terms of the Newtonian potential and the regular part $g \in  C^{\infty}(B_{\frac38}(0))$, and  observe that the contribution to $\E_\eta(v^0_\eta)$ due to the regular part is of the order $O(\eta)$:
\begin{align*}
  \eta\int_{\TTT}\int_{\TTT} v^0_\eta(x)\, v^0_\eta(y)\, G(x,y)\, dxdy
    &= \eta \int_{\TTT}\int_{\TTT} \left[\frac{1}{4\pi|x-y|} + g(x-y)\right] 
            v^0_\eta(x)\, v^0_\eta(y)\, dxdy  \\
       &= \eta^{-5} \int_{\Om_\eta}\int_{\Om_\eta} \frac{1}{4\pi|x-y|} \, dxdy + \mathcal{O}(\eta) \\
       &= \int_{\trns{\Om}_\eta}\int_{\trns{\Om}_\eta} \frac{1}{4\pi|x-y|} \, dxdy + \mathcal{O}(\eta)
\end{align*}
Consequently, we may relate the energies before and after blowup via
\beqn\label{eqn:energyeq}
\E_\eta(v^0_\eta) = \cG(\widehat\Om_\eta) - \int_{\TTT} v^0_\eta\, \rho\, dx + \bigO(\eta).
\eeqn
To prove the lemma we employ concentration-compactness methods to show that, if $\widehat\Om_\eta$ is not tight, then it must split into a finite number of components, each of which minimizes the blowup energy $\cG$.  To locate a first component, we observe that $|\trns{\Om}_\eta|=M-o(1)$ for all $\eta>0$, and apply Proposition~2.1 of \cite{FrLi2015} to obtain a set $\Om_1 \subset \RRR$ of positive measure and finite perimeter, and a sequence of points $\{y_\eta^1\}_{\eta>0} \subset \RRR$ such that
	\beqn \label{eq:LF1} \left\{
		\begin{gathered}
				\chi_{\trns{\Om}_\eta}(\ \cdot \, + y_\eta^1) \to \chi_{\Om_1} \text{ in } L^1_{\loc} \text{ as } \eta\to 0,\ \ 0<|\Om_1| \leq \liminf_{\eta\to 0}|\trns{\Om}_\eta|, \text{ and}\\
				\int_{\RRR} |\nabla \chi_{\Om_1}| \leq \liminf_{\eta\to 0} \int_{\RRR} |\nabla \chi_{\trns{\Om}_\eta}|.
		\end{gathered}\right.
	\eeqn

If it turns out that $|\Om_1|=M$, then we need go no further, and the lemma holds with $n=1$.  Indeed, let
	\[
		\tld{\Om}_\eta \,:=\, \trns{\Om}_\eta - y_\eta^1.
	\]
Since $\chi_{\tld{\Om}_\eta}\to \chi_{\Om_1}$ in $L^1_{\loc}(\RRR)$ and $\{\chi_{\tld{\Om}_\eta}\}_{\eta>0}$ is a tight sequence, in fact we have convergence globally in $L^1$ norm.  
% Also $\int_{\RRR} |\nabla \chi_{\tld{\Om}_\eta}| = \int_{\RRR} |\nabla \chi_{\trns{\Om}_\eta}| \leq C$ for $\eta>0$ small. Then there exists $\tld{\Om}\subset \RRR$ with $|\tld{\Om}|=M$ and a subsequence of $\eta$ such that $\chi_{\tld{\Om}_\eta} \wlimit \chi_{\tld{\Om}}$ weakly in $L^1(\RRR)$ as $\eta\to 0$ along this subsequence. 
 Then using the fact that $\tld{\Om}_\eta + y_\eta^1= \trns{\Om}_\eta$ and the translation invariance of the total variation and $H^{-1}$-norms, and Step 1, we conclude that
		\begin{align*}
		\cG(\Om_1) &\leq \liminf_{\eta\to 0} \cG(\widehat\Om_\eta) \\
		  &= \liminf_{\eta\to 0}\left[ \E_\eta(v^0_\eta) + \int_{\TTT} v^0_\eta\, \rho\, dx\right] 
		  \\
		  &\leq \liminf_{\eta\to 0}\left[ \E_\eta(v_\eta) + \int_{\TTT} v_\eta\, \rho\, dx\right] 
		  \\
		 &\leq e_0(M),
		\end{align*}
using the upper bound from the first order limit. Thus $e_0(M)$ is attained by $\chi_{\Om_1}$, and we may conclude that (with $r=3\delta$ and $x_\eta=\eta y_\eta^1 \in \TTT$,)
$$  \eta^3 v^0_\eta(\eta x+ x_\eta)\chi_{B_{r/\eta}(0)} = \chi_{\tld{\Omega}_\eta}(x) \to \chi_{\Omega_1} \quad\text{in $L^1_{\loc}(\R^3)$}. $$
The desired conclusion \eqref{eq:one} then follows by a change of variables from $\RRR$ to $B_r(0)\subset\TTT$, and $\eta^{-3}|\Xi_\eta|\to 0$ from Step 1.
\medskip

In the remainder of the proof we may therefore assume that $0<|\Omega_1|<M$.

\medskip

\noindent \textit{\textbf{Step 3.}} When $0<|\Omega_1|<M$ we expect there is no unique choice for the set $\Om_1$ and translations $y_\eta^1$ for which we obtain local convergence.  Following Lions \cite{Lions84} (see also Nam and van den Bosch \cite{NaBo17}) we define
$$  \M(\{\widehat\Om_\eta\}):= \sup\{ |\Om|: \ \text{there exist $y_\eta\in\RRR$ and $\Om\subset\RRR$ such that $\chi_{\widehat\Om_\eta-y_\eta}\to\chi_{\Om}$ in $L^1_{loc}(\RRR)$}\}.
$$
We may then choose $\Om_1$ and $y_\eta^1\in\RRR$ such that $|\Om_1|>\frac12 \M(\{\widehat\Om_\eta\})>0$, for which \eqref{eq:LF1} holds.
Define 
	\[
		m^1 \,:=\, |\Om_1| \qquad \text{ and } \qquad x_\eta^1 := \eta\, y_\eta^1. 
	\]
By Lemma~2.2 of \cite{FrLi2015} there exist $\{r_{1,\eta}\}_{\eta>0} \subset (0,\infty)$ such that the sets
	\[
		U_\eta^1 \,:=\, \tld{\Om}_\eta \cap B_{r_{1,\eta}}, \qquad V_\eta^1 \,:=\, \tld{\Om}_\eta \cap (\RRR \setminus \ol{B}_{r_{1,\eta}})
	\]
satisfy
	\beqn 
		%\label{eqn:splitting_properties}
			\chi_{U_\eta^1} \to \chi_{\Om_1} \text{ in } L^1(\RRR),\qquad \chi_{V_\eta^1} \to 0 \text{ in } L^1_{\loc}(\RRR),\qquad \text{and}\qquad 
			m^1_\eta:= |U_\eta^1| \to |\Omega_1|=m^1.
			\nonumber
	\eeqn	
Let $\widehat\Om^1_\eta:= V^1_\eta+y^1_\eta$, with 
$|\widehat\Om^1_\eta|= |\widehat\Om_\eta|-m_\eta^1 = M-m^1 + o(1)\in (0,M)$.
Moreover, 
$$  \lim_{k\to\infty}\left(\int_{\RRR}|\nabla\chi_{\widetilde\Om_{\eta}}|-\int_{\RRR} |\nabla\chi_{U_\eta^1}| -\int_{\RRR}|\nabla\chi_{V_\eta^1}|\right)=0,
	\quad \text{ and }\quad \liminf_{\eta\to 0} \int_{\RRR} |\nabla\chi_{U_\eta^1}| \geq \int_{\RRR}|\nabla\chi_{\Om_1}|.
$$
Using Lemma~2.3 in \cite{FrLi2015}, the splitting also extends to the nonlocal energy,
 \[
	\int_{\trns{\Om}_\eta}\!\int_{\trns{\Om}_\eta} \frac{1}{|x-y|}\,dxdy 
	= \int_{U_\eta^1}\!\int_{U_\eta^1} \frac{1}{|x-y|}\,dxdy + 
	\int_{V_\eta^1}\!\int_{V_\eta^1} \frac{1}{|x-y|}\,dxdy + o(1)
		\]
	and
		\[
			\int_{U_\eta^1}\!\int_{U_\eta^1} \frac{1}{|x-y|}\,dxdy
			 = \int_{\Om_1}\!\int_{\Om_1} \frac{1}{|x-y|}\,dxdy + o(1).
		\]
Thus, the blowup energy splits up to an error of order $o(1)$,
\beqn\label{eq:firstsplit}
\cG(\widehat\Om_\eta) = \cG(U^1_\eta) + \cG(V^1_\eta) + o(1) = \cG(\Om_1) + \cG(\widehat\Om^1_\eta) + o(1).
\eeqn 
To return to the original scale, define the functions
	\begin{gather} \label{eq:udef}
	u_\eta^1(x) \,:=\, \frac{1}{\eta^3} \chi_{U_\eta^1} \left(\frac{x-x_\eta^1}{\eta}\right)
	=\frac{1}{\eta^3} \chi_{U_\eta^1} \left(\Phi_{\eta,\delta_\eta}(x;x_\eta^1)\right) , \quad
		w_\eta^1(x) \,:=\, \frac{1}{\eta^3} \chi_{\Om_1} \left(\frac{x-x_\eta^1}{\eta}\right), \\ v_\eta^1(x) \,:=\, \frac{1}{\eta^3} \chi_{V_\eta^1} \left(\frac{x-x_\eta^1}{\eta}\right)= \frac{1}{\eta^3} \chi_{\widehat\Om^1_\eta} \left(\frac{x}{\eta}\right).
		\nonumber
	\end{gather}
Then,  
$v^0_\eta \, = \, u_\eta^1 \,  + \, v_\eta^1$, each supported in $B_{3\delta}(0)\subset\TTT$, and
applying \eqref{eqn:energyeq} to \eqref{eq:firstsplit} we obtain the following splitting,
$$   \E_\eta(v^0_\eta) = \E_\eta(u_\eta^1) + \E_\eta(v_\eta^1) + o(1)\geq 
    \E_\eta(w_\eta^1) + \E_\eta(v_\eta^1) + o(1).  $$

\medskip

\noindent \textit{\textbf{Step 4.}}  We now repeat Step 3 for the remainder term $\{v_\eta^1\}_{\eta>0}$, with associated blowup set $\widehat\Om^1_\eta$, of mass $|\widehat\Om^1_\eta|=M-m^1_\eta=M-m^1+o(1)\in (0,M)$.  
Using the compactness results in \cite{FrLi2015} as in Step 3, we obtain translations $y^2_\eta\in\RRR$, a limiting set $\Omega_2$ with $|\Omega_2|\in (0,M-m^i]$ and $|\Omega_2|>\frac12\M(\{\widehat\Om^1_\eta\}),$ and a new remainder set $\widehat\Om^2_\eta$, for which the energy splits to within an error of $o(1)$.  We iterate this procedure to create sequences of translations, limiting sets, and remainder sets, as long as the limit sets do not exhaust the total mass $M$.  More specifically, given the remainder $v^{i-1}_\eta=\eta^{-3}\chi_{\widehat\Om^{i-1}_\eta}(x/\eta)$, 
we obtain translations $y^i_\eta\in\RRR$ and $\Om_i\subset\RRR$, with $|\Om_i|>\frac 12 \M(\{\widehat\Om^{i-1}_\eta\})$, 
and a disjoint decomposition
$(\widehat\Om^{i-1}_\eta-y^i_\eta)= U^i_\eta \cup V^i_\eta$, with 
$r_{i,\eta}\in (0,\infty)$, $U^i_\eta:=(\widehat\Om^{i-1}_\eta-y^i_\eta)\cap B_{r_{i,\eta}}$, 
$V^i_\eta:=(\widehat\Om^{i-1}_\eta-y^i_\eta)\setminus B_{r_{i,\eta}}$, and radii $r_{i,\eta}$ chosen such that
$\chi_{U^i_\eta}\to \chi_{\Om_i}$ in $L^1(\RRR)$ and $\chi_{V^i_\eta}\to 0$ in $L^1_{\loc}(\RRR)$.  We set 
\begin{gather*}
  m_\eta^i:=|U^i_\eta|\to |\Om_i|:=m^i, \quad \widehat\Om^i_\eta:=V^i_\eta+y^i_\eta,\quad
     x^i_\eta:= \eta y^i_\eta, \\
     u^i_\eta:= {1\over\eta^3}\chi_{U^i_\eta}\left({x-x^i_\eta\over\eta}\right), \quad
    w^i_\eta:= {1\over\eta^3}\chi_{\Om_i}\left({x-x^i_\eta\over\eta}\right), \quad
      v^i_\eta:= {1\over\eta^3}\chi_{\widehat\Om^i_\eta}\left({x\over\eta}\right).  
\end{gather*}
Then, as in Step 3, we may assert that for any $k\geq1$:
\begin{gather}
\label{i1}
|\Om_k|>\frac 12 \M(\{\widehat\Om^{k-1}_\eta\}); \\
\label{i2}
M=\sum_{i=1}^k m^i_\eta + |\widehat\Om^k_\eta| +o(1)= \sum_{i=1}^k m^i + |\widehat\Om^k_\eta| + o(1); \\
\label{i3}
\lim_{\eta\to 0} \left\| v_\eta - \sum_{i=1}^k w_\eta^i - v_\eta^k\right\|_{L^1(\TTT)}
   =0= \lim_{\eta\to 0} \left| \widehat\Om_\eta \triangle \left( \widehat\Om_\eta^k\cup \bigcup_{i=1}^k \Om_i\right) 
       \right| ;  \\
\label{i4}
\liminf_{\eta\to 0} \left[ \cG(\widehat\Om_\eta) - \sum_{i=1}^k \cG(\Om_i) - \cG(\widehat\Om^k_\eta)\right] \geq 0; \\
\label{i4b} 
\liminf_{\eta\to 0} \left[ \E_\eta(v_\eta) - \sum_{i=1}^k \E_\eta(w^i_\eta) - \E_\eta(v^k_\eta)\right] \geq 0.
\end{gather}  
Since for each $i<k$, $(U^k_\eta + (y^k_\eta-y^i_\eta))\subset V^i_\eta\to 0$ locally, but $\chi_{U^k_\eta}\to \chi_{\Om_k}$ in $L^1(\RRR)$, we must have $|y^i_\eta - y^k_\eta|\to\infty$, that is, 
$$  {|x^i_\eta-x^k_\eta|\over\eta}\to\infty, \qquad i\neq k
$$
Note that this countable process might stop after finitely many iterations:  If $|\Om_k| = M - \sum_{i=1}^{k-1} m^i$ after $k$ iterations, then as in Step 2 we have $\chi_{\Om_\eta^k}\to \chi_{\Om_k}$ in $L^1(\RRR)$ and we may take $V_\eta^i,\Om_i=\emptyset$ for all $i>k$.  As we shall see in Step 7, this is indeed the case. 

\noindent \textit{\textbf{Step 5.}}  The above process exhausts all the mass, and yields a lower bound on the energy.
Indeed, passing to the $k\to\infty$ limit in \eqref{i2} we observe that $\sum_{i=1}^\infty m^i<\infty$, and hence from \eqref{i1} we must have $\M(\{\widehat\Om^{k}_\eta\})<2m^k\to 0$ as $k\to\infty$.  By Proposition~2.1 of \cite{FrLi2015} we may then conclude that $|\widehat\Om^k_\eta|\to 0$, and so we have strong convergence,
$$    \lim_{\eta\to 0}\left\| v_\eta - \sum_{i=1}^\infty w_\eta^i\right\|_{L^1(\TTT)}
=\lim_{\eta\to 0} \left\| v^0_\eta - \sum_{i=1}^\infty w_\eta^i\right\|_{L^1(\TTT)}= 0.  $$
Finally, from \eqref{i2} and \eqref{i4} we may conclude
\beqn\label{i5}
\cG(\widehat\Om_\eta)\geq \sum_{i=1}^\infty \cG(\Om_i) + o(1), \qquad 
   M=\sum_{i=1}^\infty m^i.
\eeqn
Using \eqref{eqn:energyeq} and $\rho(x)\leq\rho_{\max}$, we may then conclude that
\beqn\label{i6}
\E_\eta(v_\eta)\geq\E_\eta(v^0_\eta)+o(1) = \cG(\widehat\Om_\eta) - \int_{\TTT}v^0_\eta\rho\, dx + o(1)
     \geq \sum_{i=1}^\infty \cG(\Om_i) - M\rho_{\max} + o(1).
     \eeqn

\medskip

\noindent \textit{\textbf{Step 6.}} The next assertion is that $e_0(m^i)$ is attained for each $i=1,2,\ldots$. Let $\Om_i$ be as above. Then combining the lower bound \eqref{i6} with \eqref{eqn:min_conv_min} and the fact that $e_0(M) \leq \sum_{i=1}^\infty e_0(m^i)$ we obtain
	\beqn \label{eqn:sum_of_e0_mi}
		\begin{aligned}
			\sum_{i=1}^\infty e_0(m^i) - M\rho_{\max} & \geq e_0(M) - M\rho_{\max} 			\geq \liminf_{\eta\to 0} \E_\eta (v_\eta)  \\ 									& \geq \sum_{i=1}^\infty \cG(\Om_i) - M \rho_{\max} 
		\geq \sum_{i=1}^\infty e_0(m^i) - M\rho_{\max}.
		\end{aligned}
	\eeqn
Adding $M\rho_{\max}$ to each term, and noting that every term in these infinite sums are positive, we conclude that $e_0(m^i)=\cG(\Om_i)$; since $|\Om_i|=m^i$, they attain the minimum value and each $\{m^i\}_{i\in\N}\in\mathcal{M}$.  Moreover, we note that
\beqn\label{i7}   e_0(M)=\sum_{i=1}^\infty e_0(m^i).  
\eeqn 

\bigskip

\noindent \textit{\textbf{Step 7.}}  There are only finitely many components. Let $m_{c_0}>0$ be the constant given in Theorem~\ref{thm:e0}. That is, $e_0$ is attained uniquely by a ball of volume $m$ for $m\leq m_{c_0}$. Then an explicit calculation shows that
	\[
		e_0(m) = \Bigg( 3(4\pi/3)^{1/3}\Bigg)\,m^{2/3} + \Bigg((3/4\pi)^{5/3} \int_{B_1(0)}\!\int_{B_1(0)} \frac{1}{4\pi|x-y|}\,dxdy\Bigg)\, m^{5/3}.
	\]
Computing the Coulomb term of a unit ball we see that $e_0^{\pr\pr}(m) < 0$ if $m<\trns{m}:=\min\{m_{c_0},2\pi\}$. Now suppose the first two terms $m^1$ and $m^2$ of the infinite sequence $\{m^i\}_{i=1}^\infty$ are in the interval $(0,\trns{m})$. Then
	\[
		\frac{d^2}{d\eps^2}\Bigg|_{\eps=0} \Bigg( e_0(m^1+\eps) + e_0(m^2-\eps) \Bigg) = e_0^{\pr\pr}(m^1) + e_0^{\pr\pr}(m^2) < 0.
	\]
However this contradicts the fact that the sequence $\{m^1,m^2,\ldots\}$ is optimal for \eqref{eqn:e0_linearity}. Therefore there is at most one $m^i$ on the interval $(0,\trns{m})$. Since $\sum_{i=1}^\infty m^i = M <+\infty$ the number of nonzero $m^i$ has to be finite.

This concludes the proof of Lemma~\ref{lem:compactness}.
\end{proof}

\bigskip

\begin{remark}   The same compactness result could be obtained within the context of Lions' concentration-compactness \cite{Lions84}, without recourse to the results of \cite{FrLi2015}.  Indeed, by defining a L\'{e}vy concentration function which is rescaled by $\eta$,
$$   Q_{v_\eta}(t) = \eta^3\sup_{x\in\TTT} \int_{B_{\eta t}(x)} v_\eta(y)\, dy  $$
we may proceed as in Lemma~I.1 of \cite{Lions84}.  However, Proposition~2.1 of \cite{FrLi2015} effectively eliminates the vanishing case, and the following Lemmas incorporate dichotomy and compactness in a very convenient form for BV spaces of characteristic functions.
\end{remark}

\begin{remark}\label{rem:CC} \ We observe that $\{v_\eta\}$ need not be absolute minimizers of $\E_\eta$:  the same conclusions may be drawn for any sequence $v_\eta=\eta^{-3}\chi_{A_\eta}$ with $\lim_{\eta\to 0}\int_{\TTT} v_\eta=M$, and assuming an upper bound of the form,
$$   \limsup_{\eta\to 0} \E_\eta(v_\eta) \leq e_0(M) - M\rho_{\max}.  $$
\end{remark}

As a last remark, we note that thanks to Steps 6 and 7, we obtain more detailed information on the sets $U_\eta^i, V_\eta^i$, $i=1,\dots,n$, which form the basis for the decomposition.  This more refined description of these sets will be useful in proving regularity of the minimizing sequences.
\blemma\label{lem:UV}
Let $U_\eta^i, V_\eta^i$ be as in the proof of Lemma~\ref{lem:compactness}.  There exists $R>0$, independent of $\eta$, such that:
$$  U_\eta^i\subset B_{2R}(0), \qquad V_\eta^i\cap B_R=\emptyset,\quad 
\text{and}\quad \chi_{U_\eta^i\setminus B_{R/2}(0)}\to 0 \ \text{in} \ L^1(\RRR).  
$$         
\elemma

\begin{proof}
By Step 6, $\Omega_i$ attains the minimum of $e_0(m^i)$, and hence is compact and connected.  By Step 7, there are only finitely many, so we may choose $R>0$ with $\Omega_i\subset B_{R/2}(0)$ for each $i=1,\dots,n$.  Returning to Steps 3 and 4, we may now choose the radii $r_{i,\eta}\in (R, 2R)$ (see the proof of \cite[Lemma 2.3]{FrLi2015},) which yields the desired result.
\end{proof}

%%%%%%%%%%%%%%%%%%%%%%%%%%%%%%%%%%%%%%%%%%%%%%%%%%%%%%%%%%%%%%%%%%%%%%%%%%%%%%%%%%%%%%%%%%%%%%%%%%%
%%%%% Regularity of minimizers
%%%%%%%%%%%%%%%%%%%%%%%%%%%%%%%%%%%%%%%%%%%%%%%%%%%%%%%%%%%%%%%%%%%%%%%%%%%%%%%%%%%%%%%%%%%%%%%%%%%
\section{Regularity of minimizers of $\E_\eta$}\label{sec:regularity}

In determining the concentration structure of minimizers in Lemma~\ref{lem:compactness} we obtain a rough lower bound on the energy with a possible error of $o(1)$, which is enough to identify each component as a minimizer of the liquid drop energy (at length scale $\eta$), but it not sharp enough to compute the interaction between droplets.  We use the regularity theory of minimizers in order to refine this decomposition, and in particular the lower bound inequality \eqref{eqn:o1lower}, in order to obtain an optimal lower bound with no error term to compete with the interdroplet interaction energy (which we will see later is of order $\bigO(\eta^{2/3})$.)

We start with a definition which gives a property of perimeter functionals that, as we will show, provides a strong regularity result for minimizers $v_\eta$.

\begin{defi}[$\omega$-minimality] \label{def:omega_min}  Let $\mathcal{O}\subset\RRR$ be an open set, and $\omega>0$.  A set of finite perimeter $A\subset \RRR$ is an $\omega$-minimizer for the perimeter functional $\int_{\RRR}|\nabla \chi_A|$ in $\mathcal{O}$ if for any ball $B_r(x_0)\subset\mathcal{O}$ and any set of finite perimeter $B\subset\RRR$ such that $A \triangle B \subset\!\subset B_r(x_0)$ we have
	\[
	\int_{\mathcal{O}} |\nabla \chi_{A}| \leq \int_{\mathcal{O}} |\nabla \chi_B| + \omega r^3.
	\]
\end{defi}

Our goal in this section is to show that the minimizers $v_\eta$ of $\E_\eta$ blow up to $\omega$-minimizers of perimeter in $\RRR$, with $\omega$ independent of $\eta$.  
To state this result we need to blow up the minimizing set $A_\eta=\supp v_\eta$, truncated to a neighborhood of a point $p\in\TTT$, to a set in $\RRR$,
$$    \widetilde A_{p,\eta}:=\Phi_{\eta,\frac14}(A_\eta; p)=\eta^{-1}\left[(A_\eta-p)\cap B_{\frac14}(0)\right],  $$
so $\widetilde A_{p,\eta}\subset\RRR$ and centered at the origin.

\blemma \label{lem:omega_min}  
	There exists $\omega>0$, such that for any fixed $p\in\TTT$ and $R>0$, $\widetilde A_{p,\eta}$ is an $\omega$-minimizer of perimeter in $B_R(0)$, uniformly for all $\eta\in (0,{1\over 2R})$.
\elemma

We recall that $R$ is given in Lemma~\ref{lem:UV}, such that the limit sets $\Om_i\subset B_{R/2}$ for each $i$.

In proving Lemma~\ref{lem:omega_min} it is convenient to replace the mass constraint $\int_{\TTT}v_\eta = M$ with a penalization in the energy which effectively enforces this constraint, but allows us to choose comparison functions with arbitrary mass.  Following \cite{AcFuMo13}, for any $\lambda>0$ we define an unconstrained functional 
	\[
		\E_\eta^\lambda(u) \,:=\, \E_\eta (u) + \lambda \left| \int_{\TTT} u\,dx - M \right|
	\]
over $u\in BV(\TTT;\{0,1\})$. which penalizes deviations from the usual mass constraint.  Then we have:

\blemma\label{lem:penalized}
There exists constants $\eta_0,\ \lambda_0>0$ such that for every $0<\eta<\eta_0$
	\[
		\inf\big\{ \E_\eta^{\lambda_0}(u) \colon u\in BV(\TTT;\{0,1\})  \big\} =\E_\eta^{\lambda_0}(v_\eta)=\E_\eta(v_\eta).
	\]
\elemma

We first prove Lemma~\ref{lem:omega_min} assuming Lemma~\ref{lem:penalized}, whose proof follows immediately afterwards.

\begin{proof}[Proof of Lemma~\ref{lem:omega_min}]
Let $R>0$ be fixed, and take any $r>0$ with $B_r(x_0)\subset B_R(0)$, and $\tilde S\subset\RRR$ be any set satisfying 
$\tilde S \triangle \widetilde A_{p,\eta}\!\subset B_r(x_0)\subset B_R(0)$. 
We now pull back the set $\tilde S$ to $\TTT$ to construct a competitor $S_\eta\subset\TTT$ for the minimization problem for $\E^{\lambda_0}_\eta$ as follows:  inside $B_{\frac14}(p)$, we take $S_\eta$ with $S_\eta\cap B_{\frac14}(p)=\Phi^{-1}_{p,\eta}(\tilde S)$, and outside that neighborhood, $S_\eta\cap (B_{\frac14}(p))^c = A_\eta$.  Thus, $S_\eta$ coincides with $A_\eta$ except in the preimage
$\Phi^{-1}_{p,\eta}(B_r(x_0))\subset B_{\eta R}(p)$, that is
$$  S_\eta\triangle A_\eta \subset \Phi^{-1}_{p,\eta}(B_r(x_0))\subset B_{\eta R}(p),  $$
and by rescaling,
$$  |S_\eta\triangle A_\eta| = \eta^3 |\tilde S\triangle \widetilde A_{p,\eta}|.  $$

Define $\varphi_\eta:=\eta^{-3}\chi_{S_\eta}$  By Lemma~\ref{lem:penalized},  
$\E_\eta(v_\eta) = \E_\eta^{\lambda_0}(v_\eta) \leq \E_\eta^{\lambda_0}(\varphi_\eta)$.  Thus,
	\begin{align}\label{eq:um}
		\eta\int_{\TTT} |\nabla v_\eta| &\leq \eta\int_{\TTT} |\nabla \varphi_\eta| 
		+\eta\int_{\TTT}\int_{\TTT} \left[\varphi_\eta(x)\varphi_\eta(y)-v_\eta(x)v_\eta(y)\right]
		    G(x,y)\, dxdy  \\
		  &\qquad\quad+ \int_{\TTT} \rho(x)\left[v_\eta(x)-\varphi_\eta(x)\right]dx 
		  +\lambda_0\left|\int_{\TTT}\phi_\eta - M\right| . \nonumber
	\end{align}
We now estimate each term involving the difference $(\varphi_\eta-v_\eta)$, which we recall is supported in $B_{\eta R}(p)$.  For the nonlocal term, we must split the integral by decomposing $\TTT$ into $B:=B_{2\eta R}(p)$ and $B^c:=\TTT\setminus B_{2\eta R}(p)$.  When both $x,y\in B$, we may decompose the Green's function as in \eqref{eq:Green1}, and so
\begin{align*}
&\eta\int_B\int_B \left[\varphi_\eta(x)\varphi_\eta(y)-v_\eta(x)v_\eta(y)\right]
		    G(x,y)\, dxdy \\
	&\qquad= \eta\int_B\int_B \left[\varphi_\eta(x)\varphi_\eta(y)-v_\eta(x)v_\eta(y)\right]
		    \left( {1\over 4\pi |x-y|} + g(x,y)\right)\, dxdy \\
	&\qquad\leq  \int_{\tilde S} \int_{\tilde S} {1\over 4\pi |x-y|} \, dxdy -
	    \int_{\widetilde A_{p,\eta}} \int_{\widetilde A_{p,\eta}} {1\over 4\pi |x-y|} \, dxdy
	    + \eta^{-5}
	           \int_{B}\int_{S_\eta\triangle A_\eta} |g(x,y)| \, dxdy
	           \\
	 &\qquad\leq C_1 |\tilde S\triangle \widetilde A_{p,\eta}| + C_2\eta^{-5}|B_{2\eta R}(p)| \, |S_\eta\triangle A_\eta|
	    \leq (C_1+\bigO(\eta)) |\tilde S\triangle \widetilde A_{p,\eta}|,
\end{align*}
where the difference of the Newtonian potential terms is bounded using  \cite[Proposition 2.3]{BoCr14}.
For $x\in B^c$ but $y\in B$ (or vice-versa,) we recall that $\varphi_\eta(x)=v_\eta(x)$, and $S_\eta\triangle A_\eta\subset B_{\eta R}(p)$, and thus $\dist_{\TTT}(x,y)\geq R\eta$.  Since $G(x,y)$ satisfies a uniform estimate,
$$   |G(x,y)|\leq {C\over \dist_{\TTT}(x,y)} + C\,  $$
with constant independent of $x\neq y\in\TTT$, we have
\begin{align*}
&\left|\eta\int_B \int_{B^c} \left[\varphi_\eta(x)\varphi_\eta(y)-v_\eta(x)v_\eta(y)\right]
		    G(x,y)\, dxdy\right| \\
&\qquad = \left|\eta\int_B \int_{B^c} v_\eta(x)\left[\varphi_\eta(y)-v_\eta(y)\right] 
     G(x,y)\, dxdy\right|  \\
&\qquad  \leq \eta^{-2}  \int_{B^c}v_\eta(x)\left[ \int_{S_\eta\triangle A_\eta} C\left[{1\over \dist_{\TTT}(x,y)}+1\right] dy\right] dx \\
&\qquad \leq C_3 \eta^{-3} |S_\eta\triangle A_\eta| = C_3 |\tilde S \triangle \widetilde A_{p,\eta}|.
\end{align*}
Finally, if $x,y\in B^c$, then $\left[\varphi_\eta(x)\varphi_\eta(y)-v_\eta(x)v_\eta(y)\right]=0$, so the remaining integral vanishes identically, and we obtain the estimate,
\beqn\label{eq:nlterm}
  \eta\int_{\TTT}\int_{\TTT} \left[\varphi_\eta(x)\varphi_\eta(y)-v_\eta(x)v_\eta(y)\right]
		    G(x,y)\, dxdy \leq C_4 |\tilde S \triangle \widetilde A_{p,\eta}|.  
\eeqn

The remaining terms in the difference of the energies \eqref{eq:um} are easier to bound.  Indeed,
\begin{gather*}
  \int_{\TTT} \rho(x)\left[v_\eta(x)-\varphi_\eta(x)\right]dx \leq 
        \eta^{-3}\rho_{\max}|S_\eta\triangle A_\eta| = \rho_{\max}|\tilde S \triangle \widetilde A_{p,\eta}|, \\
        \text{and}\quad
        \lambda_0\left|\int_{\TTT}\phi_\eta - M\right| = \lambda_0\eta^{-3}\bigl|\, |S_\eta| - |A_\eta|\, \bigr|
          \leq  \lambda_0 |\tilde S \triangle \widetilde A_{p,\eta}|.
\end{gather*}
Since $\tilde S \triangle \widetilde A_{p,\eta}\subset B_r(x_0)\subset B_R(0):=\mathcal{O}$, from the above estimates applied to \eqref{eq:um} we have:
\begin{align*}
\int_{\mathcal{O}} |\nabla \chi_{\widetilde A_{p,\eta}}| - \int_{\mathcal{O}} |\nabla \chi_{\tilde S}|
&= \eta \int_{B_{\eta R}(p)} |\nabla v_\eta| -  \eta \int_{B_{\eta R}(p)} |\nabla \varphi_\eta| \\
&= \eta \int_{\TTT}\left( |\nabla v_\eta| - |\nabla \varphi_\eta|\right)  \\
&\leq C_5 |\tilde S \triangle \widetilde A_{p,\eta}|\leq \omega r^3,
\end{align*}
with constant $C_5=(C_4+\rho_{\max}+\lambda_0)$ and $\omega=C_5 |B_1|$, independent of $\eta>0$.
\end{proof}

\begin{proof}[Proof of Lemma~\ref{lem:penalized}] We observe that by \cite[Proposition 2.7]{AcFuMo13}, we may assume the existence of $\lambda=\lambda(\eta)$ for which the minimizer of the penalized problem coincides with the minimizer of the constrained problem, and so it suffices to show that $\lambda=\lambda_0$ may be chosen independently of $\eta$. 

Suppose that the conclusion of the lemma is not true. That is, there exist sequences $\eta\to 0$, $\lambda_\eta\to\infty$, and  $\varphi_\eta=\eta^{-3}\chi_{S_\eta}$ such that $\eta^{-3}|S_\eta|\neq M$ and
	\[
		\E_{\eta}^{\lambda_\eta}(\varphi_\eta) < \E_{\eta}^{\lambda_\eta}(v_\eta)
		 = \E_{\eta}(v_\eta).
	\]
Since $\lambda_\eta\to\infty$, we have that $\eta^{-3}|S_\eta|\to M$.  By the above upper bound and Remark~\ref{rem:CC}, Lemma~\ref{lem:compactness} applies to the family $\{\varphi_\eta\}$, and we return to the notation established there.  By Step~1, we may decompose $S_\eta=\Om_\eta\cup \Xi_\eta$ disjointly, with $\Om_\eta\subset B_{3\delta}(0)$ and $\eta^{-3}|\Xi_\eta|\to 0$.  
By Step 3 there exist translations $x_\eta^1=\eta y_\eta^1$ and disjoint sets $U_\eta^1,V_\eta^1$ so that after blowing up $\hat\Om_\eta:=\eta^{-1}(\Om_\eta-x_\eta^1)$ to $\RRR$,
$$   \hat\Om_\eta= U_\eta^1\cup V_\eta^1, \quad \chi_{U_\eta^1}\to\chi_{\Om_1} \ \text{in} \ L^1(\RRR), \quad
  \chi_{V_\eta^1}\to 0 \ \text{in} \ L^1_{\loc}(\RRR).  $$
Moreover, $|U^1_\eta|\to m^1=|\Om_1|$.  By Lemma~\ref{lem:UV} following the proof of Lemma~\ref{lem:compactness}, there exists $R>0$ so that $\Om_1\subset\!\subset B_{R/2}(0)$, $U^1_\eta\subset B_{2R}(0)$, 
$V^1_\eta\cap B_{R}(0)=\emptyset$, and $\chi_{U^1_\eta\setminus B_R(0)}\to 0$ in $L^1(\RRR)$.

Let $\mu_\eta:=M-|V_\eta^1|-\eta^{-3}|\Xi_\eta|$.  Then, $\mu_\eta= m^1_\eta +o(1)=m^1+o(1)$.  The goal is to deform the sets $U^1_\eta$ via a family of dilations such that the resulting set $\tilde U^1_\eta$ has measure exactly equal to $\mu_\eta$.  Then, we may scale back down to $\TTT$, substituting the new set $(\eta \tilde U^1_\eta + x^1_\eta)$ inside $B_{2\eta R}(x^1_\eta)$, and obtain a set with mass exactly equal to $\eta^{-3}M$, and eventually obtain a contradiction.  To that end, we define a smooth, non-increasing cut-off function of $r\geq 0$ with $0\leq \phi(r)\leq 1$, $\phi(r)=1$ for $0\leq r\leq R/2$, and $\phi(r)=0$ when $r\geq R$.  Then, we let $\Psi_t: \ \RRR\to\RRR$ be the flow map associated to the dynamical system,
$$  \mathbf{x}'(t) = \mathbf{F}(\mathbf{x}), \quad 
   \mathbf{F}(\mathbf{x})= \phi(|\mathbf{x}|)\mathbf{x}, \quad \mathbf{x}(0)=x.  $$
The right-hand side is smooth with linear growth, and so $\Psi_t$ maps $B_R(0)$ to itself for any $t\in\R$, and $\Psi_t(x)=x$ for all $x\notin B_R(0)$.  For any set $S$ of finite perimeter we define $S(t):=\Psi_t(S)$.  Then, the measure $|S(t)|$ is a differentiable function of $t$, and 
$${d\over dt}|S(t)|=\int_{S(t)}\text{div}\,\mathbf{F}\, dx 
   = \int_{S(t)} \left(3\phi(|x|)+\phi'(|x|)|x|\right)dx .
$$
As $\Om_1\subset\!\subset B_{R/2}$, there exist constants $\eps,C>0$ so that 
$${d\over dt}|\Om_1(t)|\geq 3|\Om_1(t)|\geq 2C>0, \quad\text{for all $t\in (-\eps,\eps)$.}  $$
Since $U^1_\eta\to \Om_1$ globally, we have ${d\over dt}|U^1_\eta(t)|\geq C>0$ for all
$t\in (-\eps,\eps)$, for all $\eta>0$ sufficiently small.  Recalling that $|\mu_\eta-m^1_\eta|\to 0$, we may conclude that for each sufficiently small $\eta>0$, there exists a unique $t_\eta\in (-\eps,\eps)$ so that $\tilde U^1_\eta:=U^1_\eta(t_\eta)$ has measure
$|\tilde U^1_\eta|=\mu_\eta$.  Moreover, $t_\eta=\bigO(\mu_\eta-m^1_\eta)$.  Using the expansion of the Jacobian (see \cite[Theorems 17.5,17.8]{Maggi},) we also have
\beqn\label{eq:tt}
\int_{\RRR}  |\nabla\chi_{\tilde U^1_\eta}| \leq \int_{\RRR} |\nabla\chi_{ U^1_\eta}| + Ct_\eta, \qquad 
   | \tilde U^1_\eta\triangle U^1_\eta| \leq Ct_\eta,
\eeqn
with constant $C$ depending on the perimeter of $U^1_\eta$, but independent of $\eta$.  Finally, we observe that since $\Psi_t(x)=x$ for $x\notin B_R(0)$, $\tilde U^1_\eta\cap V^1_\eta=\emptyset$.

Now we switch $\tilde U^1_\eta$ for $U^1_\eta$ and scale back down to $\TTT$.  Let $\tilde S_\eta$ be the set which coincides with $S_\eta$ in $\TTT\setminus B_{3\delta}(0)$, and with 
$\Phi^{-1}(\tilde U^1_\eta\cup V^1_\eta, y^1_\eta)$ in $B_{3\delta}(0)$, and 
$\tilde\varphi_\eta:=\eta^{-3}\chi_{\tilde S_\eta}$.  Then, by construction of $\tilde U^1_\eta$, 
$\int_{\TTT}\tilde\varphi_\eta\,dx = M$ exactly.  In addition, 
$(S_\eta\triangle \tilde S_\eta)\subset B_{2\eta R}(x^1_\eta)$, and so by the same calculation as \eqref{eq:nlterm} we have
$$   \left|\eta\int_{\TTT}\int_{\TTT}
 \left[\tilde\varphi_\eta(x)\tilde\varphi_\eta(y)-\varphi_\eta(x)\varphi_\eta(y)\right]
		    G(x,y)\, dxdy\right| \leq C_4 |\tilde U^1_\eta \triangle U^1_\eta|\leq C't_\eta,
$$
and 
$$  \left|\int_{\TTT} \rho (\tilde\varphi_\eta-\varphi_\eta)\, dx \right| \leq C't_\eta,
$$
for constant $C'$ independent of $\eta$ small, by the second estimate in \eqref{eq:tt}. 
Lastly, by the first inequality in \eqref{eq:tt} we have
$$  \left|\eta \int_{\TTT}\left( |\nabla\tilde\varphi_\eta|-|\nabla\varphi_\eta|\right)\right| \leq Ct_\eta.
$$ 
Thus, for a constant $C$ (independent of $\eta$,) we have:
\begin{align*}
  \E^{\lambda_\eta}_\eta(\tilde\varphi_\eta)-\E^{\lambda_\eta}_\eta(\varphi_\eta)
       &\leq C t_\eta - \lambda_\eta \Big| \, \int_{\TTT} \varphi_\eta \,dx - M\Big| \\
        & = C t_\eta - \lambda_\eta \bigl| \, |U^1_\eta|-|\tilde U^1_\eta| \, \bigr| \\
        & = C t_\eta - \lambda_\eta |m^1_\eta - \mu_\eta| <0
\end{align*}
for all sufficiently small $\eta>0$, as $\lambda_\eta\to\infty$ and 
$t_\eta=\bigO(m^1_\eta - \mu_\eta)$.  This contradicts the choice of $\varphi_\eta$ as the minimizer of $\E^{\lambda_\eta}_\eta$, and thus the lemma is proven.
\end{proof}

The important consequence of $\omega$-minimality is the following regularity result. 

\begin{proposition} \label{prop:omega_min}
Let $\mathcal{O}\subset\RRR$ be a bounded open set, and 
$\trns{\Om}_\eta\subset\mathcal{O}$ be a sequence of $\omega$-minimizers of the perimeter functional such that
	\[
		\sup_{\eta>0} \int_{\mathcal{O}}|\nabla\chi_{\trns{\Om}_\eta}| < +\infty \qquad \text{ and }\qquad \chi_{\trns{\Om}_\eta} \to \chi_{\Om} \quad \text{in } L^1(\mathcal{O})
	\]
for  $\Om\Subset\mathcal{O}$ of class $C^2$. Then, for $\eta$ small enough, $\trns{\Om}_\eta$ is of class $C^{1,1/2}$ and
	\[
		\pt \trns{\Om}_\eta = \{ x+\psi_\eta(x)\nu^\Om(x) \colon x\in\pt\Om \}
	\]
with $\psi_\eta \to 0$ in $C^{1,\alpha}(\pt\Om)$ for all $\alpha \in (0,1/2)$ where $\nu^\Om$ denotes the unit outward normal to $\pt\Om$.
\end{proposition}

This result appears, for example, in \cite[Theorem 4.2]{AcFuMo13}. There it is stated on the torus; however, when the target set $\Om\subset\RRR$ is bounded, using \cite[Theorem 13.8]{Maggi}, we see that the regularity theorem holds also in this context, with $\mathcal{O}=B_{2R}$ with $R$ as in Lemma~\ref{lem:UV}.

\medskip

Applying the regularity theorem to minimizers of $\E_\eta$ we may conclude that minimizers split exactly and disjointly into the sets $U^i_\eta$ found in Lemma~\ref{lem:compactness}:

\blemma\label{lem:energy_split}
Let $v_\eta=\eta^{-3}\chi_{A_\eta}$ be minimizers of $\E_\eta$ with $|A_\eta|=M\eta^3$.  Using the notation of Lemma~\ref{lem:compactness}, let $n\in\N$, and $u^i_\eta$ and $U^i_\eta\subset\RRR$, $i=1,\dots,n$, as in \eqref{eq:udef}.  Then there exists $R>0$ independent of $\eta$ such that for all sufficiently small $\eta>0$ we have:
$$
U^i_\eta\subset B_R(0), \qquad \sum_{i=1}^n |U^i_\eta| = \sum_{i=1}^n m_\eta^i = M, \quad \text{and} \
v_\eta(x) = \sum_{i=1}^n u^i_\eta(x).
$$
In addition, for sufficiently small $\eta>0$,
	\beqn\label{eqn:energy_split_exact}
		\E_\eta(v_\eta) \geq \sum_{i=1}^n \E_\eta \left( u_\eta^i \right) 
		   +\sum_{i=1\atop i\neq j}^n 
		  \eta \int_{\TTT}\int_{\TTT} {u_\eta^i\, u_\eta^j\over 4\pi |x-y|} \, dxdy + \mathcal{O}(\eta).
	\eeqn
\elemma

That is, $v_\eta$ splits into exactly $n$ well-separated components, the residual sets exactly vanish, that is $\Xi_\eta, V^n_\eta=\emptyset$  for small $\eta$, and we have a sharp lower bound on the energy.

\begin{proof}
Using the notation from the proof of Lemma \ref{lem:compactness}, we have that (after blowing up,)
	\beqn\label{eqn:hatdec}
		\widehat\Om_\eta = \bigcup_{i=1}^n \left[ U^i_\eta + y^i_\eta\right] \cup \widehat\Om_\eta^n,
	\eeqn
a disjoint union in $\RRR$, with $|y^i_\eta-y^j_\eta|\to\infty$ for $i\neq j$.  Moreover, each $\chi_{U^i_\eta}\to \chi_{\Om_i}$ in $L^1(\RRR)$, with each $\Om_i$ a minimizer of $e_0(m^i)$.  By \cite[Theorem 2.7]{BoCr14}, $\Om_i$ is precompact and $\pt \Om_i$ is of class $C^{3,\beta}$ for any $\beta<1$. Thus, there exists $R>0$ so that $\Om_i\subset B_{R/2}(0)$ for all $i=1,\dots,n$.  From the proof of Lemma~2.2 of \cite{FrLi2015}, the radii $r_\eta$ in the definition of $U^i_\eta, V^i_\eta$ may be chosen with $r_\eta\in (R, 2R)$.
By Lemma~\ref{lem:omega_min}, $\widehat\Om_\eta$ are $\omega$-minimizers of the perimeter in $B_{2R}(y^i_\eta)$, with $\omega$ independent of $\eta$, it follows from Proposition~\ref{prop:omega_min} that $\partial U^i_\eta\cap B_{2R}\subset [\partial\widehat\Om_\eta-y^i_\eta]\cap B_{2R}(0)$ is a $C^{1,\alpha}$ graph converging to $\partial\Om_i$ in $C^{1,\alpha}$-norm.  In particular, each $U^i_\eta\subset B_{R}(0)$ for all sufficiently small $\eta>0$. 

We next claim that $V_\eta^n=\emptyset$ for small enough $\eta$.  Indeed, suppose this is not the case.  Then there exists a subsequence $\eta\to 0$ (not relabelled) and point $z_\eta\in\RRR$ for which $z_\eta\in\partial V^n_\eta$.  Consider the sequence of sets $\widetilde V_\eta:=(V^n_\eta-z_\eta)$, which are uniformly $\omega$-minimizers of perimeter in $B_1(0)$, by Lemma~\ref{lem:omega_min}. 
By the compactness of $\omega$-minimizing sequences, Proposition~21.13 of \cite{Maggi}, there is a further subsequence along which $\widetilde V_\eta\to \widetilde V$ in $B_{\frac12}(0)$, and by the regularity Theorem~21.14 of \cite{Maggi}, $\widetilde V$ is itself an $\omega$-minimizer.  Therefore $\partial\widetilde V\cap B_{\frac12}(0)$ is a $C^1$ smooth surface, and by (i) of the same theorem $0\in\partial\widetilde V$.  However, $|\widetilde V\cap B_{\frac12}(0)|=\lim_{\eta\to 0}|\widetilde V_\eta\cap B_{\frac12}(0)|=0$, which is impossible given the smoothness of 
$\widetilde V\cap B_{\frac12}(0)$.  Hence we conclude that $V_\eta^n=\emptyset$ for all small $\eta$.

Finally, we recall that $v_\eta=\eta^{-3}\chi_{A_\eta}$, with $A_\eta=\Om_\eta\cup \Xi_\eta$ a disjoint union, $\Xi_\eta\cap B_{2\delta}(0)=\emptyset$, and 
$\eta^{-3}|\Xi_\eta|\leq \eta^{-3}|A_\eta\setminus B_\delta(0)|\to 0$.  As for $V_\eta^n$, we also claim that $\Xi_\eta=\emptyset$ for all sufficiently small $\eta>0$.  Indeed, assume the contrary and so there exists a subsequence (not relabelled) $\eta\to 0$ and points 
$x_\eta\in \partial\Xi_\eta\subset \TTT\setminus B_{2\delta}(0)$.  By compactness of $\TTT\setminus B_\delta(0)$ we may extract a further subsequence and $x_0\in \TTT\setminus B_{2\delta}(0)$ so that $x_\eta\to x_0$.  Consider $\Xi_\eta\cap B_\delta(x_0)$, and the blowup sets 
$$\tilde \Xi_\eta:=\Phi_{\eta,\delta}(A_\eta\setminus B_{\delta_\eta}(0);x_\eta)=\eta^{-1}[(\Xi_\eta-x_\eta)\cap B_\delta(0)]\subset\RRR.
$$
   As in the preceeding paragraph, $\tilde \Xi_\eta$ are unformly $\omega$-minimizers of perimeter in $B_1(0)$, and (by the choice of $x_\eta$,) $0\in\partial \tilde \Xi_\eta$ for each $\eta$, and so (by the regularity theorem) the sets converge and zero lies in the boundary of the limiting set.  Again, since $|\tilde \Xi_\eta|\to 0$, there is no limiting set, and by this contradiction we may conclude that 
   $\hat\Xi_\eta=\emptyset$ for all small $\eta$.

Since $v_\eta(x)=\eta^{-3}\chi_{A_\eta}(x)$ and $A_\eta=\Omega_\eta\cup\Xi_\eta$, the final identity 
$v_\eta=\sum_{i=1}^n u^i_\eta$ is verified.  As the supports of $u^i_\eta$ are well separated, 
the perimeters split exactly, 
	\[
		\int_{\TTT} |\nabla v_\eta| = \sum_{i=1}^n \int_{\TTT} |\nabla u_\eta^i| .
	\]
The nonlocal term, being quadratic, decomposes into diagonal terms, which are included in $\E_\eta(u_\eta^i)$, and off-diagonal terms,
\begin{align*}  \eta\int_{\TTT}\int_{\TTT} G(x,y) u^i_\eta(x)\, u^j_\eta(y)\, dxdy  
  &=\eta\int_{B_{3\delta}(0)}\int_{B_{3\delta}(0)} \left[{1\over 4\pi |x-y|}+ g(x,y)\right]
          u^i_\eta(x)\, u^j_\eta(y)\, dxdy  
\\
  & = \eta\int_{\TTT}\int_{\TTT} {u^i_\eta\, u^j_\eta\over 4\pi |x-y|} \, dxdy + \mathcal{O}(\eta),
\end{align*}
for $i\neq j$.  Finally, the confinement term is linear in the disjoint components, and so the lemma is proven.
\end{proof}

\medskip

\section{Proof of the lower and upper bounds}\label{sec:proof_of_conv}

Finally, in this section we use the sharp form of the decomposition to prove our main result by providing sharp upper and lower bounds on $\E_\eta(v_\eta)$.  

We begin by verifying the upper bound construction, which follows the heuristic argument given in the Introduction.

\begin{proof}[Proof of Lemma~\ref{lem:upper}]
Let $n\in\N$, points $p_1,\dots,p_n\in\RRR$, and $\{m^i\}_{i=1}^n\in\mathcal{M}$ be given.  Since $e_0(m^i)$ is attained for each $i$, there exists $z^i$ which attains the minimum in \eqref{eqn:nlip}.  Since the support of $z^i$ are bounded, and we only consider a finite number of them, there exists a constant $r>0$ for which   $Z^i:=\supp z^i \subset B_r(0)$ for all $i=1,\ldots,n$.  Let
\[
		\nu_\eta^i \,:=\, \frac{1}{\eta^3}\,z^i\left(\frac{x-\eta^{1/3}p_i}{\eta}\right),
		 \quad \text{ and }\quad \nu_\eta := \sum_{i=1}^n \nu_\eta^i.
	\]
Since $\supp\nu_\eta\subset B_{\frac14}(0)$ for all sufficiently small $\eta>0$, we also view $\nu_\eta$ as a function on $\TTT$.

Then plugging $\nu_\eta$ into $\E_\eta$ we get that
		\begin{align}\nonumber
   \E_{\eta}(\nu_\eta) &= \sum_{i=1}^n\eta \int_{\TTT} |\nabla \nu_\eta| + \eta \sum_{i,j=1}^n \int_{\TTT}\!\int_{\TTT} G(x-y)\nu_\eta^i(x)\nu_\eta^j(y)\,dxdy - 
			\sum_{i=1}^n \int_{\TTT} \nu_\eta^i(x) \rho(x)\, dx
			\\
		&=	\sum_{i=1}^n e_0(m_i) -M\rho_{\max} +
		 \eta \sum_{\substack{i,j=1\\ i\neq j}}^n \int_{\TTT}\!\int_{\TTT} {\nu_\eta^i(x)\nu_\eta^j(y)\over 4\pi |x-y|}\,dxdy  +\sum_{i=1}^n \int_{\TTT} \nu_\eta^i(x) \big( \rho_{\max}-\rho(x) \big)\,dx + O(\eta),  
		 \label{eq:up3}
		\end{align}
where we have used the representation of the Green's function in terms of the Newtonian kernel and its regular part in $B_{\frac14}(0)$, and the choice of $z^i$ as minimizers of $e_0(m^i)$.

Taking each term separately, we evaluate for $i\neq j$, via the change of variables
$\eta\xi=x-\eta^{1/3}p_i$, $\eta\zeta=y-\eta^{1/3}p_j$,
\begin{align}\nonumber
\eta\int_{\TTT}\!\int_{\TTT} {\nu_\eta^i(x)\nu_\eta^j(y)\over 4\pi |x-y|}\,dxdy
   &= \eta^{2/3}\int_{Z^i}\int_{Z^j} {1\over 4\pi|\eta^{2/3}(\xi-\zeta)-(p_i-p_j)|} d\xi\, d\zeta \\
   \label{eq:up1}
   &= \eta^{2/3} {m^i\, m^j\over 4\pi|p_i-p_j|} + o(\eta^{2/3}),
\end{align}
by Dominated Convergence applied to the integral as $\eta\to 0$.  And, in a similar way, using hypothesis (H3) on $\rho(x)$, 
\begin{align}\nonumber
\int_{\TTT} \nu_\eta^i(x) \big( \rho_{\max}-\rho(x)  \big)\,dx
&= \int_{Z^i} q\left(\eta^{1/3}[p_i+\eta^{2/3}\xi] \right) \, d\xi+ o(\eta^{2/3}) \\
\nonumber
&=  \eta^{2/3} \int_{Z^i} q(p_i+\eta^{2/3}\xi)\, d\xi + o(\eta^{2/3})  \\
\label{eq:up2}
&= \eta^{2/3} m^i q(p_i) + o(\eta^{2/3}).
\end{align}
Inserting \eqref{eq:up1}, \eqref{eq:up2} into \eqref{eq:up3} yields the desired upper bound for $n\geq 2$.
				
\medskip

For $M\in\Izero$ , we take $n=1$, $p_1=0$ and $z=\chi_{Z}$  which attains $e_0(M)$, and define
$\nu_\eta(x)=z(x/\eta)$.  Then, as in \eqref{eq:up3} we have
\begin{align*}
  \E_\eta(\nu_\eta)= e_0(M) - M\rho_{\max} +\int_Z \left[\rho(\eta\xi) -\rho_{\max}\right] d\xi + O(\eta) \\
   \leq e_0(M) - M\rho_{\max} + O(\eta), 
   \end{align*}
as the last integral is $O(\eta^2)$ (by hypothesis (H3),) but the principal error comes from the regular part of the Green's function.
\end{proof}

\bigskip

It remains to derive a matching sharp lower bound in order to complete the proof of the main result. We begin by showing that the droplet centers $\{x^i_\eta\}$ for a sequence of minimizers $v_\eta$ converge to zero at a precise rate.
	
\blemma\label{lem:points_xi}  Let $n\in\N$ be given as in Lemma~\ref{lem:compactness}.  If $n\geq 2$, then  the points $\{x_\eta^i\}_{i=1}^n$ found in Lemma \ref{lem:compactness} satisfy
		\[
			x_\eta^i = \mathcal{O}(\eta^{1/3})
		\]
	for each $i=1,\ldots,n$.  If $n=1$, then the points $y_\eta$ from Step 1 in the proof of Lemma \ref{lem:compactness} satisfy
		\[
			y_\eta = \mathcal{O}(\eta^{1/2}). 
		\]
\elemma

\begin{proof} Let $v_\eta$ be a minimizer of $\E_\eta$ with $\int_{\TTT} v_\eta\,dx=M$.
Define
			\beqn \label{eqn:alpha_beta_defn}
				\lambda_\eta \,:=\, \min_{i\neq j} |x_\eta^i - x_\eta^j| \qquad \text{ and }\qquad \beta_\eta \,:=\, \max_{i} |x_\eta^i|
			\eeqn
over finite number of indices $i,\, j=1,\ldots,n$, and let $\{u_\eta^i\}_{i=1}^n$ with 
$v_\eta=\sum_{i=1}^n u_\eta^i$ be the functions found in Lemma \ref{lem:compactness}.   Since by their construction, $|x^i_\eta-x^j_\eta|\gg\eta$, it follows that $\beta_\eta\gg \eta$.
%Define 
%$$  m^0_\eta=|\widehat\Om^n_\eta|= \eta^3\int_{\TTT} v^n_\eta\, dx = o(1).  $$

\noindent \textit{\textbf{Step 1.}}  Evaluating the nonlocal and confinement terms.  We will require the following simple estimates for the lower and upper bounds on the energy as well.  Assume $n\geq 2$.  Then, since $U^i_\eta\subset B_R(0)$, for sufficiently small $\eta$ we have
$$  \left| {1\over |x-y|} -{1\over |x^i_\eta - x^j_\eta|}\right| \leq {4R\eta\over |x^i_\eta - x^j_\eta|^2}, $$
for all $x\in \supp u^i_\eta = (\eta U^i_\eta + x^i_\eta)$ and $y\in \supp u^j_\eta = (\eta U^j_\eta + x^j_\eta)$, $i\neq j$.  
Hence,
\beqn\label{eq:nlest}
\int_{\TTT}\int_{\TTT} {u^i_\eta\, u^j_\eta\over 4\pi |x-y|} \, dxdy
    = {m^i_\eta m^j_\eta\over 4\pi |x^i_\eta - x^j_\eta|} + O(\eta|x^i_\eta - x^j_\eta|^{-2})
    = {m^i m^j\over 4\pi |x^i_\eta - x^j_\eta|}(1-o(1)).
\eeqn

To estimate the confinement term, choose $k$ for which $|x_\eta^k|=\beta_\eta$.
For $x\in (\eta U^k_\eta + x^k_\eta)$ we have $|x|\ge \frac12\beta_\eta$ and hence using (H3) on the structure of $\rho(x)$ near zero, we may conclude the rough estimate: 
$$  \rho(x)- \rho_{\max}\le -q(x)+o(|x|^2) \leq -c_2 \beta_\eta^2 ,
    $$
    for constant $c_2>0$ independent of $\eta$.
This implies that
\beqn\label{eq:conest}
\int_{\TTT} u^k_\eta\, \rho(x)\, dx \le m^k_\eta\left(\rho_{\max} - c_2\beta_\eta\right) .
\eeqn
We thus have a rough lower bound on the energy,
\begin{align}\label{eq:ienergy}
\E_\eta(v_\eta) &\geq \sum_{i=1}^n \left(e_0(m^i_\eta) -m_\eta^i\rho_{\max}\right)
   + \sum_{i,j=1\atop i\neq j}^n \eta {m^i m^j\over 4\pi |x^i_\eta - x^j_\eta|} (1-o(1))
   + c_1 m^k_\eta\, \beta_\eta^2 
    \\
   \nonumber
   &\geq e_0(M) - M\rho_{\max} + \sum_{i,j=1\atop i\neq j}^n \eta {m^i m^j\over 4\pi |x^i_\eta - x^j_\eta|} (1-o(1))
   + c_2\,\beta_\eta^2,
\end{align}
since $M=\sum_{i=1}^n m^i_\eta$, with $c_2>0$ independent of $\eta$.

\noindent \textit{\textbf{Step 2.}}  The scale of concentration, $n\geq 2$.
Let $k\neq\ell\in \{1,\dots,n\}$ be such that (along some subsequence $\eta\to 0$,) $|x^k_\eta|=\beta_\eta$, and  $|x^k_\eta-x^\ell_\eta|=\lambda_\eta$.
Matching the upper bound \eqref{eq:upper2} to the lower bound \eqref{eq:ienergy}, 
\begin{align*}  e_0(M)-M\rho_{\max} + \eta^{2/3}\mu_0 \geq \tld{\E}_\eta(v_\eta)
  &\geq e_0(M)-M\rho_{\max} + \eta {m_k m_\ell\over 4\pi |x^k_\eta-x^\ell_\eta|}(1-o(1))
     + c_2\beta_\eta^2\\
    &\geq e_0(M)-M\rho_{\max} + \eta {c_1\over \lambda_\eta}
     + c_2 \beta_\eta^2,
 \end{align*}
We may then conclude that 
\[
		\lambda_\eta \geq C_1 \eta^{1/3} \qquad \text{ and }\qquad \beta_\eta \leq C_2 \eta^{1/3}
	\]
for some constants $C_1,\,C_2>0$, as desired when $n\geq 2$.

\medskip

\noindent \textit{\textbf{Step 3.}}  When $n=1$, we have the simpler lower bound, since there are no interaction terms:
\beqn \label{eqn:LB1}
\E_\eta(v_\eta) \geq  e_0(M) - M\rho_{\max} + |y_\eta|^2 - o(\eta) = e_0(M) - M\rho_{\max} + \beta_\eta^2 - o(\eta)
\eeqn
Comparing the upper bound \eqref{eq:upper1} from Lemma~\ref{lem:upper} with the lower bound \eqref{eqn:LB1}, we have
	\[
		  e_0(M)-M\rho_{\max}+\beta_\eta^2-\mathcal{O}(\eta)\leq \E_\eta(v_\eta)\leq e_0(M)-M\rho_{\max} +\mathcal{O}(\eta)
	\]
and hence $|y_\eta|=\beta_\eta\leq \mathcal{O}(\eta^{1/2})$ as demanded.
\end{proof}

\bigskip
%%%%%%%%%%%%%%%%%%%%%%%%%%%%%%%%%%%%%%%%%%%%%%%%%%%%%%%%%%%%%%%%%%%%%%%%%%%%%%%%%%%%%%%%%%%%%%%%%%%
%%%%% Proof of Second-Order Limit Theorem
%%%%%%%%%%%%%%%%%%%%%%%%%%%%%%%%%%%%%%%%%%%%%%%%%%%%%%%%%%%%%%%%%%%%%%%%%%%%%%%%%%%%%%%%%%%%%%%%%%%

We are now ready to prove the main result stated in the Introduction.

\begin{proof}[Proof of Theorem \ref{thm:minimizers}]

Let $v_\eta\in BV(\TTT;\{0,1/\eta^3\})$ with $\int_{\TTT} v_\eta\,dx=M$ be a minimizer of $\E_\eta$ for each $\eta>0$.  Applying Lemma \ref{lem:compactness}, there exist $n\in\N$, $\{m^i\}_{i=1}^n\in\mathcal{M}$, and $\{x^i_\eta\}_{i=1}^n\subset\TTT$ satisfying the conditions given there. 
If $n=1$, then conclusion (ii) in the statement of the theorem holds, by the second part of Lemma~\ref{lem:points_xi} and \eqref{eq:one} from Lemma~\ref{lem:compactness}.

If instead $n\geq 2$, by the first part of Lemma \ref{lem:points_xi} we have that $x_\eta^i = \mathcal{O}(\eta^{1/3})$ for every $i=1,\ldots,n$. That is, the sequences $\{\eta^{-1/3} x_\eta^i\}_{\eta>0}\subset \TTT$ are bounded for each $i=1,\ldots,n$. Passing to a subsequence let us define
	 		\[
				x_i \,:=\, \lim_{\eta\to 0} \eta^{-1/3}x_\eta^i \in \TTT
			\]
for each $i=1,\ldots,n$. 
Moreover, by Lemma~\ref{lem:energy_split}, using a change of variables as in \eqref{eq:up1}, we have the precise lower bound,
\begin{align*}
 \E_\eta(v_\eta) 
 &\geq \sum_{i=1}^n \E_\eta \left( u_\eta^i \right)  
		   +\sum_{i,j=1\atop i\neq j}^n 
		   \eta\int_{\TTT}\int_{\TTT} {u^i_\eta\, u^j_\eta\over 4\pi |x-y|} \, dxdy 
		   + \mathcal{O}(\eta) \\
&\geq \sum_{i=1}^n \left(\cG(U_\eta^i) - \int_{U_\eta^i} \rho(x_\eta^i + \eta\xi)\, d\xi\right)
+\sum_{i,j=1\atop i\neq j}^n 
		  \eta\int_{U_\eta^i}\int_{U_\eta^j} {1\over 4\pi |(x^i_\eta-x^j_\eta)+ \eta(\xi-\zeta)|} 
		   d\xi\, d\eta 
		   + \mathcal{O}(\eta) \\
	&\geq \sum_{i=1}^n e_0(m^i_\eta)  - M\rho_{\max} 
	+\sum_{i,j=1\atop i\neq j}^n 
		   \eta\int_{U_\eta^i}\int_{U_\eta^j} {1\over 4\pi |(x^i_\eta-x^j_\eta)+ \eta(\xi-\zeta)|} 
		   d\xi\, d\eta 
		   \\  &\qquad \quad 
		   + \sum_{i=1}^n \int_{U_\eta^i} [\rho_{\max}-\rho(x_\eta^i + \eta\xi)]\, d\xi
		   + \mathcal{O}(\eta)
		   \\
	&\geq  e_0(M) - M\rho_{\max} 
	+\sum_{i,j=1\atop i\neq j}^n 
		  \eta^{2/3} \int_{U_\eta^i}\int_{U_\eta^j} {1\over 4\pi |\eta^{-1/3}(x^i_\eta-x^j_\eta)+ \eta^{2/3}(\xi-\zeta)|} 
		   d\xi\, d\eta 
		   \\  &\qquad \quad 
		   + \eta^{2/3}\sum_{i=1}^n\int_{U_\eta^i} q(\eta^{-1/3}x_\eta^i + \eta^{2/3}\xi)\, d\xi
		   + o(\eta^{2/3})
\end{align*}

Now, since each $\eta^{-1/3}x_\eta^i\to x_i$, and the $U_\eta^i\to \Om_i$ globally, and all sets are compact in $\RRR$, by Dominated Convergence we may conclude that
$$  \lim_{\eta\to 0}  \int_{U_\eta^i}\int_{U_\eta^j} {1\over 4\pi |\eta^{-1/3}(x^i_\eta-x^j_\eta)+ \eta^{2/3}(\xi-\zeta)|} 
		   d\xi\, d\eta = {m^i m^j \over 4\pi |x_i-x_j|}, 
$$
and
$$		   
	 \lim_{\eta\to 0} \int_{U_\eta^i} q(\eta^{-1/3}x_\eta^i + \eta^{2/3}\xi)\, d\xi =
	       m^i q(x_i),
$$
and arrive at the desired lower bound,
$$  \E_\eta(v_\eta)\geq e_0(M) - \rho_{\max}M + \eta^{2/3}\left\{
     \sum_{i,j=1\atop i\neq j}^n {m^i m^j \over 4\pi |x_i-x_j|} + \sum_{i=1}^n m^i q(x_i)\right\}
     + o(\eta^{2/3}).
$$

Matching this lower bound with the upper bound from Lemma~\ref{lem:upper}, the two match up to order $o(\eta^{2/3})$, with the $\{m^i\}$ and $\{x_j\}$ as selected by the concentration lemma.  Since the upper bound holds for any choice of masses and rescaled centers, it holds for those which minimize the quantity $\F_{m^1,\dots,m^n}$, and so those must be the values which appear in a minimizing sequence for $\E_\eta$.  Thus, the conclusion (iii) holds whenever $n\geq 2$ in Lemma~\ref{lem:compactness}.

It remains to connect the number $n$ of components with the question of whether $M\in \Izero$ or not.
First, if $M\not\in\Izero$, then  $n\geq 2$ by the last assertion in the statement of Lemma~\ref{lem:compactness}.  To complete the argument, assume $M\in \Izero$ and $n\geq 2$.  But then,
Lemma~\ref{lem:upper} provides a stronger upper bound which doesn't match with the lower bound for $n\geq 2$,
$$  e_0(M)-M\rho_{\max} + \mathcal{O}(\eta) \geq \E_\eta(v_\eta)
     \geq e_0(M)-M\rho_{\max} + \eta^{2/3}\F_{m^1,\dots,m^n} + o(\eta^{2/3}).
$$

This can only be true if $\F_{m^1,\dots,m^n}=0$, which is impossible.  Thus we must have (i) holding for $M\in \Izero$, by \eqref{eq:one} in Lemma~\ref{lem:compactness}. 
\end{proof}

%%%%%%%%%%%%%%%%%%%%%%%%%%%%%%%%%%%%%%%%%%%%%%%%%%%%%%%%%%%%%%%%%%%%%%%%%%%%%%%%%%%%%%%%%%%%%%%%%%%
%%%%% Degenerate Penalization Measures
%%%%%%%%%%%%%%%%%%%%%%%%%%%%%%%%%%%%%%%%%%%%%%%%%%%%%%%%%%%%%%%%%%%%%%%%%%%%%%%%%%%%%%%%%%%%%%%%%%%
\section{Penalization measures with degenerate maxima}\label{sec:degenerate_max}

Finally, in this section we consider the situation where the penalization measure defined via the density function $\rho$ is degenerate at its maximum value, in the sense that the Hessian degenerates at a single maximum point.  Note that by Remark \ref{rem:first_ord_general_rho} the first-order approximation still holds in such cases. 
It is at the second order that the behavior will differ, as the interaction energy will depend on the structure of $\rho$ at its maximum.

The case where $\rho$ attains its maximum at a unique point (say at $x=0$), but with vanishing Hessian, does not entail much variation from the arguments used in the nondegenerate case above.  Indeed, the concentration structure given in Lemmas~\ref{lem:compactness} and \ref{lem:energy_split} does not rely on the nondegeneracy hypothesis (H3), which only makes itself felt in the scaling arguments which determine the interdroplet distance in the upper bound construction of Lemma~\ref{lem:upper} and in the corresponding lower bound estimates of Lemma~\ref{lem:points_xi}.
In particular, we may replace the nondegeneracy condition (H3) by a homogeneity assumption of the form,
\begin{itemize} 
	\setlength\itemsep{1em}  
		\item[(H3')]  There exist $q>2$ and $r,\rho_1>0$ so that $\rho\in C^1(B_r(0))$ and 
	\[
		\rho(x) \,:=\, \rho(0) - \rho_1 |x|^q + o(|x|^q), \qquad\text{as $x\to 0$,}
	\]
\end{itemize}
without changing the concentration structure of minimizers.  However, this choice will alter the geometry of minimizing droplet configurations: the appropriate secondary scale for pattern formation occurs at distance of $\bigO(\eta^{1/(q+1)})$, and the energy scale for interaction between droplets will be of order $\bigO(\eta^{q/(q+1)}$.  That is, minimizers $v_\eta$ will have an energy expansion of the form
$$   \E_\eta(v_\eta) = e_0(M) -M\rho_{\max} + \eta^{{q\over q+1}}\, F_0(v) + o\left(\eta^{{q\over q+1}}\right),
$$
with
	\beqn 
		\F_{m^1,\dots,m^n}(v) \,:=\, \begin{dcases*}
								\rho_1 \sum_{i=1}^n m^i|x_i|^q + \frac{1}{4\pi}\sum_{\substack{i,j=1\\ i\neq j}}^n \frac{m^i\,m^j}{|x_i-x_j|} & if $v=\sum_{i=1}^n m^i\delta_{x_i}$ with $\{x_i\}$ distinct, $\{m^i\}\in\mathcal{M}$, \\
								+\infty																													& otherwise
						   \end{dcases*}
	    \nonumber
	\eeqn
and $\{m^i\}\in\mathcal{M}$, $\{x_i\}$ chosen to minimize $\F_{m^1,\dots,m^n}(v)$.

\bigskip

%%%%%%%%%%%%%%%%%%%%%%%%%%%%%%%%%%%%%%%%%%%%%%%%%%%%%%%%%%%%%%%%%%%%%%%%%%%%%%%%%%%%%%%%%%%%%%%%%%%
%%%%%%%%%%%%%%%%%%%%%%%%%%%%%%%%%%%%%%%%%%%%%%%%%%%%%%%%%%%%%%%%%%%%%%%%%%%%%%%%%%%%%%%%%%%%%%%%%%%

\section{Appendix}

Here we briefly discuss the physics behind diblock copolymers, the role of nanoparticle attraction in confinement, and the droplet scaling regime.

\subsection*{Diblock copolymers and confinement}\label{sec:scaling}

Diblock copolymers are macromolecules composed of two  chemically distinct homogeneous 
polymer chains (of monomer species A and B respectively) linked together by a covalent bond (cf. \cite{BatesFredrickson99, Fred-book}). The thermodynamical incompatibility between the different sub-chains drives the system to phase separate; however, the covalent
bonds between the different sub-chains prevent phase separation at a macroscopic length 
scale. As a result of these two competing trends, block copolymers undergo phase
separation at a nanometer length scale, leading to an amazingly rich array of nanostructures.
Perhaps the simplest model for self-assembly of diblock copolymers is via the \emph{Ohta--Kawasaki functional} \cite{OK}, a phase-field model (diffuse interface) with nonlocal interactions.  Recently, there has been considerable interest in the physics and engineering communities on self-assembly of diblock copolymers under confinement (see for example \cite{Shi} and the references therein). Of particular interest is the situation where the length of the confinement region is comparable with the length of the diblock copolymer macromolecules. 

Recently three of the authors (see \cite{AlBrTo15}) introduced and analyzed a simple model for a polymer/nanoparticle blend, wherein the nanoparticles were attracted to one of the two molecular species in the polymers.  
Copolymer/nanoparticle interaction is modeled in the Ohta--Kawasaki functional by the addition of a penalization term which prefers the $u=1$ phase in the neighborhood of each nanoparticle \cite{Ginzburg_et_al_1999}.  Assuming a fixed limiting distribution of a large number of nanoparticles occupying a vanishingly small volume of the sample, and passing to the sharp interface limit, \cite{AlBrTo15} obtain a model for diblock copolymers under confinement of the form,
\beqn\label{eqn:energy_gamma_sigma}
		\varE_{\gamma,\sigma}(u) = \int_{\TTT} |\nabla u| \,+\, \gamma \, \|u-m\|_{H^{-1}(\TTT)}^2 \,+\, \sigma \, \int_{\TTT}(1-u(x))\rho(x)\,dx, \qquad \int_{\TTT} u\,dx =m.
		\nonumber
	\eeqn
  For relatively large $\sigma>0$, the region of high nanoparticle density effectively confines the $\{ u = 1\}$ phase to this region, and we may simply view the penalization term in the energy as a confinement term subject to the density $\rho$.  Indeed, an example in \cite{AlBrTo15} illustrates how confinement modifies minimizing patterns in the isoperimetric problem ($\sigma=0$), from lamellar to a sphere, with increasing $\sigma$.

\subsection*{The Droplet scaling regime}\label{sec:Droplet}
We arrive at the desired singularly perturbed energy functional \eqref{venergy} by extending the droplet scaling of \cite{ChPe2010} to encompass the confinement term in \eqref{eqn:energy_gamma_sigma}.
First, we define an auxiliary shifted energy by
	\beqn\label{eqn:energy_shifted}
		\begin{aligned}
		\E_{\gamma,\sigma}(u) &:= \varE_{\gamma,\sigma}(u) - \sigma\int_{\TTT}\rho(x)\,dx  \\
							   &= \int_{\TTT} |\nabla u| \,+\, \gamma \, \|u-m\|_{H^{-1}(\TTT)}^2 \,-\, \sigma \, \int_{\TTT} u(x)\rho(x)\,dx.
		\end{aligned}
		\nonumber
	\eeqn

\bigskip

We now introduce the master parameter $0<\eta\ll 1$, intended to represent the approximate diameter of the droplets of the phase $\{u=1\}$.  The droplet volume will tend to zero, so we need to adjust the order parameter to produce particles of fixed mass, $M>0$.  Thus, we introduce a concentrating order parameter $v_\eta:\TTT \to \{ 0,1/\eta^3\}$ via 
$$  v_\eta(x):= \eta^{-3} u(x), \quad\text{so}\quad \int_{\TTT} v_\eta\,dx = \eta^{-3}m:= M.  $$
We note that the support of $v_\eta$ is the same as that of $u$.  
 We then have
	\beqn\label{eqn:energy_scaling}
				\E_{\gamma,\sigma}(u) = \eta^2 \left[ \eta\,\int_{\TTT} |\nabla v_\eta| \,+\, \gamma \, \eta^4   \|v_\eta - M\|_{H^{-1}(\TTT)}^2 \,-\, \sigma\eta \, \int_{\TTT} v_\eta\rho(x)\,dx \right].
		\eeqn
Given the anticipated scaling for $v_\eta$ as suggested by the ansatz \eqref{eq:ansatz}, we expect each term in the brackets to be of order one, and thus we define the droplet regime by the following choice of the parameters,
	\[
		\gamma = \frac{1}{\eta^3} \qquad\text{ and }\qquad \sigma = \frac{1}{\eta}.
	\]
Thus, we define the rescaled energy as follows:
	\beqn\label{eqn:energy_first_order}
		\begin{aligned}
		\E_{\eta}(v) &:= \frac{1}{\eta^2} \E_{\gamma,\sigma}(u) = \frac{1}{\eta^2} \left(\varE_{\gamma,\sigma}(u) - \sigma\int_\TTT\rho\right) \\
					 &= \eta\, \int_{\TTT} |\nabla v| \,+\, \eta \, \|v-M\|_{H^{-1}(\TTT)}^2 \,-\, \int_{\TTT} v(x)\rho(x)\,dx.
		\end{aligned}
	\eeqn

\bigskip

\noindent {\bf Acknowledgements.} The authors were supported by NSERC (Canada) Discovery Grants. IT was also supported by a Field--Ontario Postdoctoral Fellowship.  We are also grateful to An-Chang Shi for pointing out that the functional used in  \cite{AlBrTo15} can also be used as a model problem for diblock copolymers under confinement.

%%%%%%%%%%%%%%%%%%%%%%%%%%%%%%%%%%%%%%%%%%%%%%%%%%%%%%%%%%%%%%%%%%%%%%%%%%%%%%%%%%%%%%%%%%%%%%%%%%%
%%%%%%%%%%%%%%%%%%%%%%%%%%%%%%%%%%%%%%%%%%%%%%%%%%%%%%%%%%%%%%%%%%%%%%%%%%%%%%%%%%%%%%%%%%%%%%%%%%%
\bibliographystyle{IEEEtranS}
\def\url#1{}
\bibliography{NanopartBib}

\end{document}